\documentclass{amsart}

\usepackage{amssymb,amsmath}

\usepackage{hyperref}
\newtheorem{theorem}{Theorem}[section]
\newtheorem{lemma}[theorem]{Lemma}
\newtheorem{proposition}[theorem]{Proposition}
\newtheorem{corollary}[theorem]{Corollary}

\theoremstyle{definition}
\newtheorem{definition}[theorem]{Definition}
\newtheorem{example}[theorem]{Example}

\newtheorem{remark}[theorem]{Remark}
\theoremstyle{remark}

\numberwithin{equation}{section}

\newcommand{\wt}{\widetilde}

\newcommand{\I}{I}
\newcommand{\N}{\mathbb{N}}
\newcommand{\K}{K}
\renewcommand{\L}{\Lambda}
\renewcommand{\o}{\omega}
\renewcommand{\b}{\beta}
\newcommand{\J}{J}

\newcommand{\wh}{\widehat}
\newcommand{\e}{\varepsilon}
\newcommand{\p}{\partial}
\newcommand{\A}{\mathcal{A}}
\newcommand{\R}{\mathbb{R}}
\newcommand{\loc}{\mathrm{loc}}

\DeclareMathOperator*{\esssup}{ess\,sup}

\newcommand{\C}{\mathbb{C}}
\renewcommand{\d}{\delta}

\renewcommand{\a}{\alpha}
\newcommand{\ii}{\mathcal{S}}
\newcommand{\m}{\mu}
\newcommand{\g}{\gamma}

\newcommand{\ad}{{\operatorname{ad}}}

\renewcommand{\div}{{\mathrm{div}}}
 \renewcommand{\Im}{{\mathrm{Im}}}
 \renewcommand{\rho}{ {\varrho}}

\renewcommand{\phi}{{\varphi}}
 \renewcommand{\epsilon}{ {\varepsilon}}

\renewcommand{\r}{\varrho}
\newcommand{\s}{\sigma}
\newcommand{\abs}[1]{\left\lvert#1\right\rvert}
\renewcommand{\t}{\tau}

\begin{document}

\title[Nonsmooth  H\"ormander
 vector fields ]{Nonsmooth  H\"ormander
 vector fields   \\ and their control balls }


\author{Annamaria Montanari}
\address{Dipartimento di Matematica, Universit\`a di Bologna (ITALY)}
\email{montanar@dm.unibo.it, morbidel@dm.unibo.it}
\thanks{}

\author{Daniele Morbidelli}
\thanks{}

\subjclass[2010]{Primary 53C17; Secondary 35R03}

\date{}

\dedicatory{}

\begin{abstract}  We prove a ball-box theorem for nonsmooth  H\"ormander vector
  fields of step $s\geq 2.$
\end{abstract}

\maketitle

\tableofcontents

\section{Introduction}
In this paper we give a self-contained proof of a ball-box theorem
for a family  $\{X_1, \dots, X_m\}$ of   nonsmooth vector fields
satisfying the  H\"ormander  condition. This is the third   paper, after   \cite{M} and \cite{MM}, where we investigate  ideas of the classical
article by Nagel Stein and Wainger \cite{NSW}.

Our  purpose  is
to prove a ball-box theorem  using only elementary analysis
techniques  and  at the same time to relax as much as possible the
regularity assumptions on the vector fields.
  Roughly speaking,
our results hold as soon as the commutators involved in the
H\"ormander condition are Lipschitz continuous. Moreover,   our
proof does not rely on algebraic tools, like  formal series and the
Campbell--Hausdorff formula.

To describe our work, we recall the basic ideas of \cite{NSW}. Notation and
language are more precisely described in Section
\ref{last}.
  Any control ball $B(x,r)$ associated with a family $\{X_1,\dots,X_m\}$  of H\"ormander vector fields in $\R^n$
   satisfies, for $x$ belonging to some compact set $K$ and small radius $r<r_0 $,  the
double
  inclusion
\begin{equation}\label{iioo}
 \Phi_x(Q(C^{-1}r))\subset B(x,r)\subset  \Phi_x(Q(C r)).
\end{equation}
Here, the map
  $\Phi_x$ is an exponential  of the form
 \begin{equation}
  \label{dinage}\Phi_x(h)=
 \exp(h_1U_1+\cdots+h_n U_n)(x),
 \end{equation} where  the vector fields $U_1,\dots,U_n$ are suitable
commutators of lengths $d_1,\dots,d_n$ and
$Q (r) =   \{ h\in\R^n : \max_j |h_j|^{1/d_j}< r\}.  $
Usually, \eqref{iioo} is    referred to as  a \emph{ball-box} inclusion.
 A control on the Jacobian matrix of $\Phi_x$ gives an estimate of the measure of the ball and ultimately it provides the doubling property.

 A remarkable achievement in \cite{NSW} concerns  the
 choice of the
  vector fields $U_j$ which guarantee  inclusions \eqref{iioo}
for a given control ball $B(x,r)$, see also the discussion in
\cite[p.~440]{S}.
Enumerate  as  $Y_1,\dots, Y_q$ all commutators of length at most $ s$
 and let  $\ell_i$ be the length of $Y_i.$  If the H\"ormander condition of step $s$
 is fulfilled, then the vector fields $Y_i$ span $\R^n$ at any point.
Given a multi-index $\I=(i_1, \dots, i_n)\in \{1, \dots, q\}^n=: \ii$  and its corresponding  $n-$tuple  $Y_{i_1}, \dots ,Y_{i_n}$
of commutators, let
\begin{equation}\label{massi}
 \begin{aligned}\lambda_\I(x)& =\det (Y_{i_1}, \dots ,Y_{i_n})(x)
\quad\text{and}
\quad \ell(\I)= \ell_{i_1}+\cdots +\ell_{i_n}.
\end{aligned} \end{equation}
In \cite{NSW}, the authors prove   the following fact: given a ball
$B(x,r)$, inclusion \eqref{iioo} holds with $U_1= Y_{i_1},\dots, U_n= Y_{i_n}$ if the
 $n-$tuple $\I\in \ii $
 satisfies the \emph{$\eta$-maximality} condition
\begin{equation}\label{maxx}
 |\lambda_\I(x)| r^{\ell(\I)}>\eta \max_{K\in \ii} |\lambda_K(x)| r^{\ell({K})},
\end{equation}
where $\eta\in(0,1)$ is  greater than some absolute constant.
Although the choice of the $n$-tuple  $I$ may depend on both  the
point and the radius, the constant $C$ is uniform in $x\in K$ and
$r\in(0,r_0)$.

 In \cite{M} the second author proved that   \eqref{iioo} also  holds if one changes  the map $\Phi_x$ with
the \emph{almost exponential} map
\begin{equation}\label{maso}
 E_x (h) = \exp_*(h_1 U_1)\circ\cdots\circ  \exp_*(h_n U_n)(x),
 \end{equation}
 where    $h_j\mapsto \exp_*(h_j U_j)$ is the \emph{approximate  exponential}
 of the commutator $U_j$,    whose  main feature is that
 it can be factorized as  a suitable composition of exponentials of the original vector fields $X_1,\dots,X_m$.
 See \eqref{approdo} for the definition of $\exp_*$.  Lanconelli and the second
 author in  \cite{LM} proved that, if inclusion \eqref{maso}, with pertinent estimates for
 the Jacobian of   $E_x$ are known, then   the Poincar\'e  inequality   follows
 (see \cite{J} for the original proof).
 It is worth to observe now that
all the results in \cite{NSW} and \cite{M}   are   proved for  $C^M
$   vector fields, where $M$ is  much larger than the step $s$.
 This can be seen by carefully reading the proofs of   Lemmas 2.10 and 2.13 in  \cite{NSW}.

In \cite[Section 4]{TW}, Tao and Wright gave a new proof of the ball-box theorem
with a different approach, based  on  Gronwall's inequality.
The authors in \cite{TW} use
 scaling maps  of the form $ \Phi_{x,r}(t): = \exp(t_1
r^{d_1}U_1 + \cdots + t_n r^{d_n} U_n)x$, which are naturally
defined on a box $|t|\le \e_0$, where $\e_0>0$ is a small constant
independent of $x$ and $r$, see the discussion in Subsection
\ref{cinquantadue}.
 The arguments in \cite{TW}
do not rely on the Campbell--Hausdorff formula. \footnote{The methods of \cite{TW} have been further exploited in a very recent paper by Street \cite{Str}.}
Moreover, although the
statement is phrased  for $C^\infty$ vector fields, one can see that their
results hold under the assumption that the vector fields have a   $C^M$   smoothness, with $M=2s$  for vector
fields of step $s$.   See Remark \ref{regolarita} for
a more detailed discussion.

 In \cite{MM} we started
 to work in low regularity hypotheses and we obtained a ball-box theorem and the Poincar\'e inequality
   for
    Lipschitz continuous vector fields of step two with Lipschitz continuous commutators. We used the maps
    \eqref{maso}, but several aspects of the work \cite{MM}
 are peculiar of the step two situation and until now it was not clear how to generalize those results
to  higher step vector fields.

Recently, Bramanti, Brandolini and Pedroni   \cite{BBP}    have proved a
doubling property and the Poincar\'e
 inequality for nonsmooth H\"ormander
 vector fields with
an algebraic method. Informally speaking,
they truncate the Taylor series  of the
coefficients of the vector fields and then they apply to the   polynomial   approximations  the
 results in \cite{NSW,LM} and \cite{M}.
 The paper \cite{BBP} also involves a study of  the almost exponential maps   in \eqref{maso}.
The results in \cite{BBP} and in the present paper were obtained
independently and simultaneously.

In this paper  we complete the result in \cite{MM}, namely we prove a ball-box
theorem for general vector fields of arbitrary
 step $s$, requiring
  basically  that all the commutators involved in the H\"ormander condition are Lipschitz continuous. Our precise
 hypotheses  are stated in Definition  \ref{aesse}.
 We improve all previous results in term of regularity, see Remark
\ref{regolarita}.
As in \cite{MM}, we use  the almost exponential maps \eqref{maso}, but we need to provide a very detailed study
of such functions in the higher step case.

The scheme of the proof of
  our theorem is basically the Nagel, Stein and Wainger's one,
 but there are  some   new tools
that should be emphasized.  Namely, we obtain
   some non commutative calculus formulas  developed in order to  show that, given a commutator $Y$,
   the derivative   $\frac{d}{dt} \exp_*(t Y)   $  can be precisely written   as
   a finite sum of higher order commutators
   plus an \emph{integral remainder}.  This is done in Section
  \ref{jtk}.
  These results are applied in
Section \ref{uno4}   to the almost exponential maps $E$ in \eqref{maso}. Our main structure theorem
 is
 Theorem \ref{distro}.
As in \cite{MM}, part of our computations will be given for smooth
vector fields, namely the standard Euclidean regularization $X_j^\s$
of the vector fields $X_j$. We will keep everywhere under control
all constants in order to
 make sure that they are stable as $\sigma$ goes to $0$.

It is well known (see \cite{LM,MM}) that the doubling property and the Poincar\'e inequality
 follow immediately from  Theorem \ref{distro}.
Observe  also that our
    ball-box theorem     can be useful in all situations where
integrals
of the form    $\int  |f(x)-f(y)|w(x,y)dxdy$ need to be
estimated, for some weight $w$. See for example \cite{M} or
\cite{MoM}.   As an application,   in Proposition \ref{ormai} we prove
a  subelliptic H\"ormander--type estimate for nonsmooth
vector fields. We believe that the results in Section \ref{jtk} may be
  useful in other, related, situations.

Concerning  the machinary developed in Section
\ref{jtk}, it is worth to mention the  papers
 \cite{RS,RS2}, where  non commutative calculus formulas are used in the proof of
a nonsmooth version of Chow's Theorem  for   vector fields of step two.

Geometric analysis for nonsmooth vector fields
 started in the   80s with the papers by Franchi and Lanconelli \cite{FL1,FL2}, who proved the Poincar\'e inequality
 for diagonal vector fields in $\R^n$ of the form
 $X_j = \lambda_j(x)\p_j$, $j=1, \dots, n$.
In the  diagonal case completely different techniques are available.  In the recent paper by Sawyer and Wheeden
\cite{SW}, which probably contains the best results to date on
diagonal vector fields,  the  reader can  find a rich bibliography
on the subject.

 \medskip
 \noindent{\bf Plan of the paper.}
  In Section \ref{last}  we introduce
   notation. In Section
   \ref{jtk}
    we prove our noncommutative calculus formulas and in Section \ref{jkt}
        we prove
a stability property of the
  ``almost-maximality''  condition \eqref{maxx}. These tools  are applied in Subsection
  \ref{unotre} to the maps $E$.  In subsection \ref{cinquantadue} we briefly discuss the ``scaled version''
  of our maps $E$.
   Subsection \ref{principale}
contains the ball-box theorem. In Section \ref{esempi} we show some
examples. Finally, Section  \ref{appendiamo} contains the  smooth
approximation result for the original vector fields.

 \medskip
 \noindent{\bf Acknowledgment.
 }
   We wish to express our gratitude to Ermanno
   Lanconelli, for his continuous advice, encouragement and interest in our work, past
   and
   present. We dedicate
this paper to him with admiration.

\section{Preliminaries and notation} \label{last}
\setcounter{equation}{0} We consider    vector fields
 $
 X_1, \dots, X_m
$ in $\R^n$.    For any $\ell\in\N $ we define a \emph{word} $w=j_1\dots
j_\ell$  to be any   finite ordered
collection of $\ell$ letters,   $j_k\in\{1,\dots, m\}$, and we introduce the   notation
$ X_w  = [X_{j_1},\cdots ,[X_{j_{\ell-1}},X_{j_\ell}]]$ for
commutators.
   Let  $ |w|  := \ell$ be the \emph{length} of $X_w$.
We assume the H\"ormander condition of step $s$, i.e. that
$\{X_w(x): |w |\le s\}$  generate all $\R^n$ at any point
$x\in\R^n$. Sometimes  it  will be useful to have a different
notation  between a vector field in $\R^n$
 and its  associated vector function. In these situations we will write  $X_w= f_w\cdot\nabla =\sum_{k=1}^n
 f_w^k\p_k$.
We will also enumerate as $Y_1, \dots, Y_q$ all the commutators
$X_w$ with length $|w|\le s$ and denote by $\ell_i$ or $\ell(Y_i)$
the length of $Y_i$ . We identify an ordered $n$--tuple of
commutators $Y_{i_1}, \dots, Y_{i_n}$ by the   index
$I=(i_1, \dots, i_n)\in  \ii:=\{1,\dots,q\}^n$.

For $x,y\in\R^n$,  denote by  $d(x,y) $ the control distance, that is  the
infimum of the $r>0$
such that there is  a Lipschitz path
 $\gamma:[0,1]\to \R^n$ with $\gamma(0)=x$, $\gamma(1)=y$ and $\dot\gamma=\sum_{j=1}^m b_j X_j(\gamma)$,
  for a.e. $t\in[0,1]$. The measurable functions $b_j $  must satisfy  $|b_j(t)|\le r$ for almost any $t$.
 Corresponding balls will be indicated as $B(x,r)$.

Denote also  by $\varrho(x,y)$ the infimum of the $r>0$  such
that there is a Lipschitz continuous path $\gamma:[0,1]\to \R^n$
with $\gamma(0)=x$, $\gamma(1)=y$ and $\gamma $ satisfies for a.e.
$t\in[0,1]$, $\dot\gamma=\sum_{i=1}^q c_i Y_i(\gamma)$ for suitable
measurable functions $c_j $ with $|c_j(t)|\le r^{\ell(Y_j)}$.
Corresponding balls will be  denoted by $B_\r(x,r)$.
The definition of $\r$ is meaningful as soon as the vector fields
$Y_j$ are at least continuous.

\begin{definition}[Vector fields of class $\mathcal{A}_s$]\label{aesse}
 Let  $X_1,\dots,X_m $ be  vector fields in $\R^n$ and
 let $s\ge 2$. We say that the vector fields $X_j$ are of class $\A_s$ if
 they are of class $C^{s-2,1}_{\loc}(\R^n)$ and  for any word
   $w $ with $|w|=s-1$, and for every $j,k\in\{1,\dots,m\}$,

  (1) the derivative $X_k f_w$ exists and it is continuous;

(2) the  distributional derivative $X_j (X_k f_w)$ exists
and \begin{equation}
     \label{dhc}
X_j (X_k f_w)\in    L^\infty_{\loc}(\R^n).
    \end{equation}
\end{definition}

Recall that $X_j\in C^{s-2,1}_{\loc}$ means that  all the Euclidean derivatives of order at most $s-2$ of the functions
$f_1,\dots, f_m$ are locally Lipschitz   continuous. In particular,  all the
commutators  $X_w$, with $|w| \le  s-1$ are locally Lipschitz continuous in the
Euclidean sense and by
item (1)  all commutators
   $X_w$  of length $|w|= s$ are pointwise defined. If we knew
that $d$ defines the Euclidean topology, condition
  $(2)$ would equivalent to the fact that $X_w$
 is locally $d$-Lipschitz, if $|w|=s$, see   \cite{GN,FSSC}.

Let $\{X_1,\dots,X_m\}$ be in the class $\A_s$ and assume that they satisfy the H\"ormander condition of step $s$.  Fix once for all a pair of  bounded
  connected open sets $\Omega'\subset\subset \Omega$ and denote
   $K=\overline{\Omega'}$.
We denote by  $D$  Euclidean derivatives. If $D=
\p_{j_1}\cdots\p_{j_p} $ for some $j_1, \dots, j_p\in\{1,\dots,n\}$,
then $|D|:=p$ indicates the order of $D$.
 It is understood that a derivative of order $0$ is the identity.
 Introduce the positive
  constant
\begin{equation}
\begin{aligned} \label{lipo}
 L: & =  \max_{\substack{ 1\le j\le m \\ 0\le \abs{D}\le s-2}}
\sup_\Omega  |Df_j|
 +\max_{\substack{j=1,\dots, m\\
|D|=s-1}}\esssup_\Omega \abs{D f_j}
\\&\qquad +
 \max_{\substack{
   k,j=1,\dots,m,\\   |w|= s-1}
} \esssup_\Omega |X_kX_j f_w|.
\end{aligned}\end{equation}

\begin{remark}
We will prove in Section 5 a ball-box theorem for vector fields of
step $s$ in the class $\A_s$. This  improves both the results in
\cite{TW} and \cite{BBP} in term of regularity. Indeed, in \cite{TW}
a
 $C^M$ regularity, with $M=2s$ must be assumed (see Remark \ref{regolarita}). In \cite{BBP} the authors assume that
 the vector fields belong to the Euclidean Lipschitz space
 $ C^{s-1,1}_{\loc}(\R^n)$,
  which requires the boundedness on the Euclidean gradient $\nabla f_w$ of any commutator $f_w$ of length $s$, while
  we only need to  control only the ``horizontal" gradient of $f_w$.

\end{remark}

\subsection*{Approximate commutators} For vector fields $X_{j_1}, \dots, X_{j_\ell}$,
and for $\t>0,$ we define, as in \cite{NSW}, \cite{M} and \cite{MM},
\[
 \begin{aligned}
 C_\t( X_{j_1})& := \exp(\t X_{j_1}),
 \\ C_\t( X_{j_1}, X_{j_2})& :=\exp(-\t X_{j_2})\exp(-\t X_{j_1})\exp(\t X_{j_2})\exp(\t X_{j_1}),
 \\&\vdots
  \\C_\t( X_{j_1}, \dots, X_{j_\ell})&
:=C_\t( X_{j_2}, \dots, X_{j_\ell})^{-1}\exp(-\t X_{j_1}) C_\t( X_{j_2}, \dots, X_{j_\ell})\exp(\t X_{j_1}). \end{aligned}
 \]
Then let
\begin{equation}
e_*^{tX_{j_1 j_2\dots j_\ell}} :=  \exp_*(t X_{{j_1 j_2\dots j_\ell}}):= \left\{\begin{aligned}
& C_{t^{1/\ell}}(X_{j_1}, \dots, X_{j_\ell} ), \quad &\text{ if $t>0$,}
\\
&C_{|t|^{1/\ell}}(X_{j_1}, \dots, X_{j_\ell} )^{-1}, \quad &\text{ if $t<0$.}
                 \end{aligned}\right.
\label{approdo}
\end{equation}
By standard ODE theory,    there is $t_0$  depending on $\ell, K$,
$\Omega$,  $ \sup\abs{ f_j } $ and  $\esssup\abs{\nabla f_j}  $     such that
$\exp_*(t X_{{j_1 j_2\dots j_\ell}})x$ is well defined for any $x\in
 K$ and $|t|\le t_0$.
The approximate commutators $C_t$  are quite natural (indeed, they
make an appearance in the original paper \cite{NSW}). Assuming that the
vector fields are smooth and using
 the Campbell--Hausdorff formula,
 we have the formal
expansion
 \begin{equation*}
  C_\t(X_{j_1}, \dots, X_{j_\ell}) = \exp\Big( 
 \t^\ell X_{j_1j_2\dots j_\ell} +\sum_{k=\ell+1}^\infty \t^k R_k\Big),
 \end{equation*}
 where $R_k$ denotes a linear combination of commutators of length $k$. See \cite[Lemma 2.21]{NSW}.
 A study of these maps in the smooth case based on this formula is carried out
in \cite{M}.

   Define, given $\I=(i_1,\dots,i_n)\in\ii$,   $x\in K$ and $h\in\R^n$, with   $|h|\le C^{-1}$
     \begin{equation}
 \label{hhh}E_{I,x}(h):=E_\I(x,h):=\exp_*(h_1 Y_{i_1})\cdots \exp_*(h_n Y_{i_n})(x),
   \end{equation}
 \begin{equation*}
\|h\|_\I : =\max_{j=1,\dots,n}|h_j|^{1/\ell_{i_j} },\qquad   Q_\I(r): =\{h\in\R^n:\|h\|_\I < r\}
\end{equation*}
 \begin{equation*}
   \Lambda(x,r): =\displaystyle{\max_{{K}\in \ii} }\abs{\lambda_{K}(x)} r^{\ell({K})},
 \end{equation*}
 where $\ell(\K)=\ell_{k_1}+\cdots+\ell_{k_n}$,  the determinants $ \lambda_{K}$
are defined in \eqref{massi}, and we have
\begin{equation}\label{horma}
 \nu :=\inf_{x\in \Omega}\Lambda(x,1)>0.
\end{equation}
The lower bound \eqref{horma} will appear many times in the following sections.
All the constants in our main   theorem will depend on $\nu$ in
\eqref{horma} and on $L$ in \eqref{lipo}.

In order to refer to the crucial condition  \eqref{maxx}, we give the following definition
\begin{definition}[$\eta-$maximal triple]  \label{lader} Let $\eta\in \left]0,1\right[$,    $\I\in\ii$, $x\in\R^n$ and $r>0$. We say that $(\I,x,r)$ is $ \eta-$maximal,   if  we have
$  |\lambda_\I(x)| r^{\ell(\I)}>\eta \Lambda(x,r). $
\end{definition}

   \subsection*{Regularized vector fields.} Here we describe our procedure of smoothing of the vector fields $X_j$ of
   step $s$.
For  for any function $f$, let
  $f^{(\sigma)}(x)=\int
 f(x-\sigma y)\phi(y)dy$, where  $  \phi\in C_0^{\infty} $ is a  standard nonnegative
 averaging kernel supported in the unit ball.
 Define
 \begin{equation}
  \label{civetta}\begin{aligned}
                  X_j^\s & : = \sum_{k=1}^n(f_j^k)^{(\s)}\p_k
                \quad\text{and}\quad
\\  X_{j_1\dots j_\ell}^\sigma
&: =   [X_{j_1}^\s,\cdots,
[X_{j_{\ell-1}}^\s,X_{j_\ell}^\s]]= : \sum_{k=1}^n (f_{j_1\dots j_\ell}^k)^\s\p_k,
                 \end{aligned}
 \end{equation}
 for any word $j_1\dots j_\ell, $ with  $2\le \ell\le s$.     (Observe that $f_w^\s\neq  f_w^{(\s)}$, if $|w|>1$.
See Section \ref{appendiamo})
 Then:
 \begin{proposition}\label{schiaccia2}
 Let $X_1, \dots, X_m$ be vector fields in the class $\A_s$. Then the following hold. \begin{enumerate}
  \item For any $\ell=1,\dots, s$, for any word  $w$  of lenght $|w|\le \ell$,
  \begin{equation}\label{appo}
X_w^\sigma \to
    X_w,
  \end{equation}
as $\sigma\to 0$,   uniformly on $K$.
In particular,  for any  multi-index $I=(i_1, \dots, i_n)\in\ii$, we have
$
 \lambda_\I^{\sigma} :=\det(Y_{i_1}^\sigma , \dots, Y_{i_n}^\sigma)\longrightarrow \lambda_\I,
$
uniformly on $K$, as $\sigma\to 0$.

 \item There is $\s_0>0$ such that, if $|w|= s$ and $k=1,\dots, m$,
 then
 \begin{equation}\label{diffo}
\sup_{0<\s<\s_0} \sup_{x\in K} | X_k^\s f_w^\s|\le C,
 \end{equation}
 with  $C$ depending on $L$ in \eqref{lipo}.

 \item There is $r_0$ depending on $K,\Omega$ and the constant
  in \eqref{lipo} such that the following holds. Let $x\in K$, $r<r_0$ and  $b\in L^\infty([0,1], \R^m)$ with $
  \|b_j\|_{L^\infty}\le r$ for all $j$. Then there is a unique  $\phi\in\mathrm{Lip}([0,1],\R^n)$,  a.e. solution of
  $\dot\phi =
   \sum_j b_j X_j(\phi)$, with $\phi(0)=x$. Denote also by  $\phi^\s\in\mathrm{Lip}([0,1],\R^n)$, the  a.e. solution
   of the
  $
    \dot \phi^\s=\sum_{j} b_j X_j^\s(\phi^\s)$,   with $\phi^\s(0)=x$. Then
 \begin{equation}
 \phi^\s(1) \to \phi(1),
  \end{equation}
  as $\s\to 0$, uniformly in $x\in K$.
 As a consequence, for any $\I\in\ii$, uniformly in  $x\in K$, $|h|\le C^{-1}$,
 \begin{equation}\label{appone}
E_\I^\s(x,h):=\exp_*(h_1 Y_{i_1}^\s)\cdots \exp_*(h_n Y_{i_n}^\s)\to E_\I(x,h).
\end{equation}
 \end{enumerate}
 \end{proposition}
\begin{proof}  The proofs of items 1 and 2 are given in details in Section \ref{appendiamo}.    Item 3 follows from standard
    properties of ODE. \end{proof}

\begin{remark}
The approximation result contained in Proposition \ref{schiaccia2} is crucial for our subsequent arguments.
Note that   the  class $\A_s$ requires a control on the Euclidean gradients of
 all commutators of length strictly less than $s$.
However, it is natural to conjecture that a control only along the
horizontal directions could be sufficient to ensure   our main
structure theorem in Section \ref{uno4}. Unfortunately, it seems
quite difficult to get an approximation theorem as Proposition
\ref{schiaccia2} for a more general class than $\A_s$. On the other
side,
 working without mollified vector fields seems to rise
 some non trivial  new issues which we plan to face in a further study. \end{remark}

 \subsection*{Some more notation.} Our notation for constants are the following:
     $C, C_0  $ denote large absolute constants, $\e_0,r_0, t_0,
C^{-1}$
      or $C_0^{-1}$ denote  positive small  absolute
      constants. ``Absolute
       constants'' may depend on the dimension $n$, the number $m$ of the fields,  their step $s$, the
     constant $L$ in \eqref{lipo} and possibly the constant $\nu$  in  \eqref{horma}.
  We also use the notation $\e_\eta$ (or $C_\eta$)
   to denote a small (or a large) constant depending also on $\eta$.
The constants $\s_0$ or $\wt\sigma$ appearing in the regularizing
parameter $\sigma$ may also depend on the Euclidean continuity
moduli of the vector fields $f_w$, with $|w|=s,$ which are not
included in $L.$ Composition of functions are shortened as follows:
$fg$ stands for $f\circ g$.
 The notation  $u$  is always used for   functions of  the form $\exp(t_1
Z_1)\cdots\exp(t_\nu Z_\nu)$ for some $t_j\in\R$, $ \nu \ge 1$, $
Z_j\in \{ X_1,\dots, X_m\}$.

\section{Approximate exponentials of commutators}\label{jtk}
\setcounter{equation}{0}
The main result of this section is Theorem \ref{dididi} in Subsection \ref{tazzurella},
 where we prove an exact formula for the derivative
$  \frac{d}{dt}u(e_*^{tX_w}(x))$, where $X_w$ is a
commutator of length $|w|\le s$, while $e_*$ is the approximate exponential defined in \eqref{approdo}.
All this section is written for smooth vector fields, namely  the  mollified $X_j^\s$, but all constants are
appearing
 in our computations are stable as $\s$ goes to
 $ 0$. We drop everywhere in this section  the superscript $\s$.

 We will  show that
 \begin{equation} \label{integ} \begin{aligned}
  \frac{d}{dt} e_*^{tX_w}(x) = X_w (e_*^{tX_w}(x) )&+\text{ higher order commutators }
 \\& + \text{ integral remainder. }
 \end{aligned}\end{equation}
 The integral remainder is rather  complicated, but we do not need its exact form.
 In order to
  understand
  what we need to compute the derivative in \eqref{integ}, let us try to   calculate for example
 the derivative  $\frac{d}{dt} u(e^{tX} e^{tY} x)$, where $X, Y\in \{ \pm X_1, \dots, \pm X_m\}$ and   $u$ denotes  the identity
  function in $\R^n$. Since $X$ and $Y$ are $C^1$, we have
\begin{equation*}
 \frac{d}{dt}  u(e^{tX} e^{tY} x) = (Xu) (e^{tX} e^{tY} x) + Y  (ue^{tX} )(e^{tY} x).
\end{equation*}
In order to compare  the terms in the right-hand side, we may write
\begin{equation*}
 Y  (ue^{tX} )(e^{tY} x) = Y  u(e^{tX} e^{tY} x) + \int_0^t\frac{d}{d\t}  Y  (ue^{\t X})(e^{-\t X} e^{tX} e^{tY} x)
 d\t.
\end{equation*}
  Lemma \ref{seminuovo}
 below shows that the derivative inside the integral can be written in an exact form in term of the commutator of $X$ and $Y$.
The purpose of the following Subsection \ref{reste}
is to establish a formalism  to study in a precise way more general, related, integral expressions.

\begin{lemma}\label{seminuovo} Let $Z,X$ be smooth vector fields. Then,
\begin{equation}\label{lachi}
 \frac{d}{d
 t}Z (u e^{-t X})(e^{t X}y)=[X,Z](u e^{-t X})(e^{t X}y).
\end{equation}
\end{lemma}

\begin{proof} The lemma is known but we provide a proof for completeness.
 Observe first that
\[
 \begin{aligned}
\frac{d}{dt}Z(ue^{-tX})(e^{tX}x)&=  \frac{d}{d\t} Z(ue^{-tX})(e^{\t X}x) \Big|_{\t=t}+
\frac{d}{d\t} Z(ue^{-\t X})(e^{t X}x) \Big|_{\t=t}
\\&=:(1)+(2).
 \end{aligned}
\]
Obviously, $(1)=XZ  (ue^{-tX})(e^{t X}x)$. Write now $(2)$ as follows
\[
 \begin{aligned}
 \frac{d}{dt} Z(ue^{-tX})(\xi)\big|_{\xi=e^{tX}x}&= \frac{d}{dt} Z^j
 (\xi)\p_{\xi_j}(ue^{-tX})(\xi)\big|_{\xi=e^{tX}x}
\\&= Z^j (\xi)\p_{\xi_j}\frac{d}{dt}(ue^{-tX})(\xi)\big|_{\xi=e^{tX}x} \end{aligned}
\]
The proof of formula \eqref{lachi} will be concluded as soon as we prove that
\begin{equation}\label{caspo}
\frac{d}{dt}(ue^{-tX})(\xi) = -X(ue^{-tX})(\xi).
\end{equation}
To prove \eqref{caspo},  start from
the identity $u(\eta)= u(e^{-tX}e^{tX}\eta)$, for small $t$. Differentiating,
\[
 \begin{aligned}
  0&=\frac{d}{dt} (u(e^{-tX}e^{tX}\eta)) =  \frac{d}{d\t} (u(e^{-tX}e^{\t X}\eta))|_{\t=t}+ \frac{d}{d\t} (u(e^{-\t
  X}e^{tX}\eta))|_{\t=t}
  \\&= Z (ue^{-tX})(e^{tX}\eta)+ \frac{d}{d\t} (u(e^{-\t
  X}e^{tX}\eta))|_{\t=t}.
 \end{aligned}
\]
Then, \eqref{caspo} is proved by letting $e^{tX}\eta=\xi$.
\end{proof}

\subsection{Notation for integral remainders} \label{reste}  Let $\lambda\in\N$, $p\in\{2,\dots,s+1\}$. We denote, for  $y\in K$,  and    $t\in[0,t_0]$, $t_0$ small enough,
\begin{equation}\label{nota}
 O_{p }(t^\lambda,u,y)= \sum_{i=1}^N\int_0^t
 \o_i(t,\t )  X_{w_i}(u\phi_i^{-1} e^{-\t Z_{i}})( e^{\t Z_{i}}\phi_i y)d\t,
\end{equation}
where $N$ is a suitable integer and  $u$ is the identity map or
$u=\exp(t Y_1)\cdots \exp(t Y_\mu),$ for some  integer $\mu$ and suitable vector
 fields $Y_j\in \{ \pm X_1 ,\dots,\pm X_m\}$.
Here $X_{w_i}$ actually stands for a mollified $X_{w_i}^\s$, but we
drop the superscript for simplicity. To describe the generic  term
of the sum  above, we drop the dependence on $i$:
\begin{equation}\label{tipicissimo}
(R):= \int_0^t
 \o(t,\t)   X_w(u\phi^{-1}e^{-\t  X})( e^{\t X}\phi y)d\t .
\end{equation}
  Here $X_w$ is a commutator of  length   $|w|=p $ and  $X\in \{ \pm X_j\}$. Moreover, for
 any $t<t_0$,
 the function  $\o(t,\t)$ is a
 polynomial, homogeneous  of degree $\lambda -1$ in  all variables $(t,\t)$, such that $\o(t,\t ) > 0$ if
 $0<\t<t.$  Thus
\begin{equation}
 \label{opp}
 \int_0^t \omega (t,\t ) d\t =  b t^\lambda\quad \text{ for any }  \;t>0,
\end{equation}
for a suitable constant $b\in\R$.
  The map $\phi$  is the identity map or
$\phi= \exp(t  Z_1)\cdots \exp(t  Z_\nu)$
for some $\nu\in\N$, where    $Z_j\in\{ \pm X_1,\dots,\pm X_m\}$.

\begin{remark}\label{costola} All the numbers $N, \mu,\nu,  b, $ appearing in the computations of this paper will be bounded by  absolute constants.
  \end{remark}

In order to explain how this  formalism works, we give the main properties of our   integral remainders.

\begin{proposition} \label{ossera}    A remainder of the form \eqref{nota} satisfies for every $\alpha\in\N$
\begin{equation}\label{laprop}
 t^\a O_p{(t^\lambda,u,y})= O_p({t^{\a+\lambda},u,y}) \quad \text{ for all } y\in K\quad  t\in[0, t_0].
\end{equation}
Moreover, for $p\le s+1$,
\begin{equation}\label{prie}
| O_{p }(t^\lambda,u,y)|  \le C t^\lambda   \qquad \text{ for all }y\in K\quad t \in[0, t_0],
\end{equation} where $t_0$ and  $C$ depend on the constant $L$ in \eqref{lipo} and on the numbers
 $N,\mu, \nu,  b$ appearing in the sum \eqref{nota}.
 Furthermore, if $\ell(Z)=1$ and  $p\le s+1$,
\begin{equation}\begin{aligned}
            O_p(t^\lambda, ue^{tZ}, y)= O_p(t^\lambda, u, e^{tZ} y).
                \end{aligned}
\label{note2}
\end{equation}
Finally, if $p\le s$, we may write, for suitable constants   $c_w$, $|w|=p,$
\begin{equation}\label{iii} \begin{aligned}
 O_{p }(t^\lambda,u,y)
  &= \sum_{|w|=p}c_w t^{\lambda}X_w u(y) + O_{p +1}(t^{\lambda+1} ,u,y) .
\end{aligned}
\end{equation}
 \end{proposition}
 \begin{proof}  The proof of \eqref{laprop} and   \eqref{note2} are rather easy
  and   we leave them to the reader.
 So we start with the proof of \eqref{prie}. A
 typical term in $O_{p }(t^\lambda,u,y)$ has the form
\begin{equation}\label{occhio}
 \int_0^t \omega(t,\t)    Y(u\phi^{-1} e^{-\t Z})(e^{\t Z}\phi y)d\t ,
\end{equation}
with $\ell(Y)= p\le s+1$.  Thus, by Proposition \ref{schiaccia2}, we have
$\abs{ Y(u\phi^{-1} e^{-\t Z})(e^{\t Z}\phi y)}\le C$ (observe that we need
 \eqref{diffo}, if $p=s+1$).
Therefore, \eqref{prie} follows from  the property \eqref{opp} of   $\omega$.

  Finally we establish  the key property \eqref{iii}.
    Start from the generic term of $O_p(t^\lambda, u, y)$ in \eqref{occhio},
where we  introduce the notation $g_k:= e^{t  Z_k}\cdots e^{t  Z_\nu} $,
for $k=1,\dots, \nu$ and $g_{\nu+1}$ denotes the identity map. Recall also that $\ell(Y)\le s$.
Therefore,
  we get
 \[
  \begin{aligned}
&\int_0^t  \o(t, \t)   Y(u e^{-t  Z_\nu}\cdots e^{-t Z_1} e^{-\t X})\big( e^{\t X}
e^{t Z_1 }\cdots e^{t  Z_\nu}y\big)  d\t
\\&=\int_0^t  \o(t,\t)
  Y(u  g_1^{-1} e^{-\t X})\big( e^{\t X} g_1 y\big)  d\t
 \\&\quad =
\int_0^t  \o(t,\t) Yu(y)   d\t
  +
  \int_0^t  \o(t,\t) \Big\{ Y(u  g_1^{-1}  e^{-\t X})\big(
   e^{\t X}
 g_1 y\big)-Yu(y)\Big\}  d\t
\\&\quad =bYu(y)t^\lambda
 +\int_0^t \o(t, \t)
   \Big\{Y(u  g_1^{-1}e^{-\t X})\big( e^{\t X}
 g_1 y\big)
-
Y(u  g_1^{-1} )\big(
 g_1 y\big) \Big\} d\t
 \\&\qquad \qquad
    +  \sum_{k=1}^{\nu}
 \int_0^t \o(t, \t)  \Big\{Y(u  g_k^{-1})\big(
 g_k y\big)
-
Y(u  g_{k+1}^{-1})\big(
 g_{k+1}y\big)
\Big\}  d\t.
  \end{aligned}
 \]
Recall  that $Y$ has length $p\le s$.
 The   penultimate term   can be written as
 \begin{equation*}
\begin{aligned}
  \int_0^t  & \o(t,\t ) \Big\{Y(u  g_1^{-1}  e^{-\t X})\big(
   e^{\t  X}
 g_1 y\big)-Yu
(u  g_1^{-1})\big(
 g_1 y\big)\Big\}  d\t
\\& =
 \int_0^t   d\t \; \o(t,\t) \int_0^\t  d\s \frac{d}{d\s}  Y(u  g_1^{-1}  e^{-\s X})\big(
   e^{\s X} g_1 y)
   \\&=\int_0^t d\s \Big\{\int_\s^t \o(t,\t) d\t  \Big\}
   [X,Y](u  g_1^{-1}  e^{-\s X})\big(
   e^{\s X} g_1 y) .
 \end{aligned}
 \end{equation*}
 Observe that, as required,
  the function $\wt\omega (t,\s):= \int_\s^t \o(t,\t) d\t$ satisfies
\begin{equation*}
 \int_0^t \wt\omega(t, \sigma) d\s = \int_{0}^t d\t \o(t, \t)\int_0^\t d\s =
 \int_{0}^t d\t \; \t\o(t, \t) = \wt b t^{\lambda+1},
 \end{equation*}
because $(t, \t)\mapsto \t\o(t, \t)$ is homogeneous of degree $\lambda$.

 The $k-$th term  in the sum has the form
 \[\begin{aligned}
  \int_0^t  &d\t \, \o(t,\t )   \int_0^{t }d\s\frac{d}{d \s}
  Y(u g_{k+1}^{-1} e^{-\s Z_{k}})\big(
e^{\s Z_{k}} g_{k+1} y\big)
\\&=\int_0^t d\s \,\wt \omega(t, \s)  [Z_k, Y]
(u g_{k+1}^{-1} e^{-\s Z_{k}})\big(
e^{\s Z_{k}} g_{k+1} y\big),
  \end{aligned}
 \]
 where $
\wt  \omega(t, \s) :=\int_0^t\o(t,  \t) d\t = b t^\lambda
$ has the correct form. The proof is concluded.
 \end{proof}

\subsection{Higher order non commutative calculus
formulas}\label{higher}
 In order
to prove Theorem \ref{poiche}, we first need to iterate formula
\eqref{lachi}. Start from smooth vector fields $X :=X_j^\s$
 of length one and    $Z:=X_w $ of length $\ell(Z):=|w|$.
 Differentiating
  identity \eqref{lachi} we get, by the Taylor formula
\begin{equation*}
Z (u e^{-t X})(e^{tX} y)= \sum_{k=0}^r\frac{t^k}{k!}\ad_X^k Zu(y)
 +\int_0^t\frac{(t-\t)^r}{r!} \ad_X^{r+1}Z(u e^{-\t X})(e^{\t X} y)d\t,
\end{equation*}
 where we introduced the notation:
$\ad_XZ = [X, Z]$, $\ \ad^2_X Z= [X,[X,Z]]$, etcetera.
In other words,
\begin{equation}\label{skg}
\begin{aligned}Z &(u e^{t X})(  y)  -Z u ( e^{t X} y)
\\& = \sum_{k=1}^r\frac{t^k}{k!}\ad_{-X}^k Zu(e^{tX} y)
 +\int_0^t\frac{(t-\t)^r}{r!} \ad_{-X}^{r+1}Z(u e^{\t X})(e^{-\t X}e^{tX} y)d\t
\\& =\sum_{k=1}^r\frac{t^k}{k!}\ad_{-X}^kZu(e^{tX} y) +O_{r+1+\ell(Z)}(t^{r+1}, u,
e^{tX}y).
\end{aligned}\end{equation}
If we take  $r=s-\ell(Z)$,  we may write
\begin{equation}\label{skg2}
\begin{aligned} Z (u e^{t X})(  y) & -Z u ( e^{t X} y)
\\&
= \sum_{k=1}^{s-\ell(Z) } \frac{t^k}{k!}\ad_{-X}^kYu(e^{tX} y) +O_{s+1}(t^{s-\ell(Z) +1}, u, e^{tX}y).
\end{aligned}\end{equation}
 In view of \eqref{prie},
 this order of expansion is the highest which ensures that the remainder can be estimated with $C t^{s-\ell(Z)+1}$,
 with
 a control on $C$ in term of the constant in \eqref{lipo}, as soon as $y\in K$ and $|t|\le C^{-1}$.

 Next, we seek for a family of higher order  formulas,
 in which we change $e^{tX}$ with an approximate exponential  $\exp_*(t X_w)$.
The coefficients of the expansion \eqref{skg} are all explicit  but
we do not need such an accuracy in the higher order formulae.
 To explain what suffices for our purposes, start with the case of commutators of length two.  Let
  $C_t=C_t(X,Y)=e^{-tY}e^{-tX}e^{tY}e^{tX}$, where $X:=X_j^\s $  and $ Y:= X_k^\s$ are mollified
vector fields with length one.

Let $Z:=X_v^\s $  be a smooth  commutator
with length $\ell(Z):=|v|$.
Assume first that $\ell(Z)=s$. Then, iterating \eqref{skg2} we can write
\begin{equation*}\tag{$F_{2, 1}$}
 Z (uC_t)(x)  -Zu(C_t x)= O_{s+1}(t, u, C_t x) .
\end{equation*}
If instead   $\ell(Z)=s-1$, then some elementary computations based on
   \eqref{skg} give
\begin{equation*}\tag{$F_{2, 2}$}
   \begin{aligned}Z & (uC_t)(x)  -Zu(C_t x)
   \\& =\sum_{k_1+ k_2+ k_3+ k_4=1}
      \frac{t^{k_1+\cdots k_4}}{k_1!\cdots k_4!}\ad_Y^{k_4}\ad_X^{k_3}\ad_{-Y}^{k_2} \ad_{-X}^{k_1} Z u(C_t x)
   \\&\qquad \qquad   + O_{2+\ell (Z)}(t^2, u, C_t x)
      \\&
     =   O_{2+\ell (Z)}(t^2, u, C_t x)=
  O_{s+1}(t^2, u, C_t x).
     \end{aligned}
\end{equation*}
Next, if
  $\ell(Z)=s-2$, (this can happen only if $s\ge 3$),
 then we must expand more. Namely, we have
\begin{equation*}\tag{$F_{2, 3}$}
  \begin{aligned}
   Z & (uC_t)(x) -Zu(C_t x)
\\&
=     \sum_{k_1+k_2+k_3+k_4=1}^{2}
    \frac{t^{k_1+\cdots k_4}}{k_1!\cdots k_4!}\ad_Y^{k_4}\ad_X^{k_3}\ad_{-Y}^{k_2} \ad_{-X}^{k_1}Z u(C_t x)
    \\&\qquad \qquad  +O_{3+\ell(Z)}(t^3, u, C_t x)
\\&     = t^{2}[Z, [X,Y]]u(C_t x) +O_{3+\ell(Z)}(t^3, u, C_t x)
\\&= t^{2}[Z, [X,Y]]u(C_t x) +O_{s+1}(t^3, u, C_t x).
\end{aligned}
\end{equation*}
\label{aggiorno} Finally, if  $ \ell(Z)\le s-3$ (this requires at least $s\ge4$),
 we must expand even more:
  \begin{equation*}\tag{$F_{2, 4}$}
   \begin{aligned}
 &   Z(uC_t)(x) -Zu(C_t x)
    \\&= \hspace{-3 ex}\sum_{k_1+k_2+k_3+k_4=1}^{3}
    \frac{t^{k_1+\cdots k_4}}{k_1!\cdots k_4!}\ad_Y^{k_4}\ad_X^{k_3}\ad_{-Y}^{k_2} \ad_{-X}^{k_1}Z u(C_t x)
     +O_{4+\ell(Z)}(t^4, u, C_t x)
   \\&= t^{2}[Z, [X,Y]]u(C_t x) + t^3\Big\{
   \frac 1 2\ad^2_Y\ad_X Z u(C_t x) -\frac12\ad_Y\ad^2_X Zu(C_t x)
   \\&\quad -\frac12 \ad^2_X\ad_YZu(C_tx)+\frac12\ad_X\ad^2_Y Z u(C_t x)
   \\&\quad -\ad_Y\ad_X\ad_Y Z u(C_t x) +\ad_X\ad_Y\ad_X Zu(C_t x)\Big\}
   +O_{4+\ell(Z)}(t^4, u, C_t x)
.\end{aligned}
  \end{equation*}
Observe that if $\ell(Z)=s-3$, then $O_{4+\ell(Z)}(t^4, u, C_t x) =O_{s+1}(t^4, u, C_t x)$. If instead $\ell(Z)<s-3$, then
we can expand up to the order $O_{s+1} (t^{s+1-\ell(Z)}, u, C_t x) $ by means of  \eqref{iii}.

We have started to  put tags of the form   $(F_{k, \lambda})$
in our formulae. The  number $k$ indicates the length of the commutator we are approximating,
   while the   number $\lambda$ denotes the power of $t$ which controls   the remainder.

Note that   in  $(F_{2,4})$,
    the curly bracket changes sign if we exchange $X$ with $Y$. Briefly, we can write  \begin{equation*}
  \begin{aligned}
   Z (uC_t)(x) &=
   Z u(C_tx) + t^2 [Z, [X, Y]]u(C_tx)
\\&\quad   +    t^3\hspace{-2ex}\sum_{|w|=3+\ell(Z)} \hspace{-2ex}c_w X_{w}u(C_t x) + O_{4+\ell(Z)} (t^{4}, u, C_t x).
  \end{aligned}
 \end{equation*}
 for all $x\in K$, $t\in [0, C^{-1}]$, where the coefficients $c_w$ are determined in $(F_{2, 4})$.
 The corresponding formula for $C_t^{-1}(X,Y)$ is
 \begin{equation*}
  \begin{aligned}
   Z(uC_t^{-1})(x) &=
   Z u(C_t^{-1}x)
   - t^2 [Z, [X, Y]]u(C_t^{-1}x)
   \\&\qquad
   +   t^3\hspace{-1ex}\sum_{|w|=3+\ell(Z)}  \hspace{-1ex}\wt c_w X_{w}u(C_t^{-1} x)
    + O_{4+\ell(Z)} (t^{4}, u, C_t^{-1} x),
  \end{aligned}
 \end{equation*}
where, since $C_t^{-1}(X,Y)=C_t(Y,X)$,
 the coefficients
  $\wt c_w$
  are obtained again from $(F_{2, 4})$, by changing $X$ and $Y$. We are not interested in the explicit knowledge
 of all  the coefficients $c_w $ and $\wt c_w$. We only need to observe  the following
 remarkable  cancellation property:
  \[\sum_{|w|=3+\ell(Z)}(c_w+\wt c_w)X_w(x)=0 \quad\text{for any $x\in K$.}\]

Next we generalize formulae $(F_{2,\lambda})$ above.
 The general statement we prove  tells
  that   this   cancellation persists when  the length of the commutator we are approximating with
  $C_t$ is three or more.
\begin{theorem}\label{grossetto}
For any  $\ell\in\{2, \dots, s\}$, $x\in K$, $t\in [0, C^{-1}]$, the following
family $(F_{\ell, 1}, F_{\ell, 2}, \dots, F_{\ell, s})$ of formulas holds.

\smallskip\noindent{\emph{Formulas  $F_{\ell,k}$.}}
  For any  $C_t= C_t(X_{w_1}, \dots, X_{w_\ell})$, $k=1,\dots,\ell$
   and   for any commutator $Z $ such that
   $\ell(Z)+k\le s+1$, we have
   \begin{equation*}
  \begin{aligned}
   Z(uC_t)(x)-Z u(C_tx) &= O_{k+\ell(Z)}(t^k,u, C_t x)
\\
    Z(uC_t^{-1})(y)-Z u(C_t^{-1}y) &=O_{k+\ell(Z)}(t^k, u, C_t^{-1} x).  \end{aligned}
 \end{equation*}
 \noindent {\emph{Formula $F_{\ell,\ell+1}$.}}  Let $\ell\ge 2$  be such that   $\ell+1\le s$. Then,
for any  $C_t(X_{w_1}, \dots, X_{w_\ell})$ and    $Z$ such that $\ell+1+\ell(Z)\le s+1$,
 \begin{equation*}
  \begin{aligned}
  & Z(uC_t)(x)  -Z u(C_tx) =t^\ell[Z, X_{w}]u(C_t x)  + O_{\ell+1+\ell(Z)}(t^{\ell+1},u,C_t x),
\\&
 Z(uC_t^{-1})(y)  -Z u(C_t^{-1}y)
   = -t^\ell[Z, X_{w}]u(C_t^{-1} x)  + O_{\ell+1+\ell(Z)}(t^{\ell+1},u,C_t^{-1}x).
  \end{aligned}
 \end{equation*}
  \noindent {\emph{Formula $F_{\ell,\ell+2}$.}} 
If $s\ge 4$,  let $\ell\ge 2$ be such that   $\ell+2\le s$. Then,
  for any  $C_t(X_{w_1}, \dots, X_{w_\ell})$ and    $Z$ such that
   $\ell+2+\ell(Z)\le s+1$,   there are numbers  $c_v, \wt
  c_v $, with $|v|=\ell+\ell(Z)+1$, such that
\begin{equation}\label{brandi}
\begin{aligned}
                 Z(uC_t)(x)-Zu(C_tx)& =
                  t^\ell[Z, X_w]u(C_t x) +t^{\ell+1}
                  \hspace{-2ex}\sum_{|v|=\ell+\ell(Z)+1}\hspace{-2ex}c_v X_v u(C_tx)
\\&\qquad + O_{\ell+2+\ell(Z)}(t^{\ell+2}, u, C_t x)
\\             Z(uC_t^{-1})(x)-Zu(C_t^{-1}x) & =
-t^\ell[Z, X_w]u(C_t^{-1} x) +t^{\ell+1}\hspace{-2ex}\sum_{|v|=\ell+\ell(Z)+1}\hspace{-2ex}\wt c_v  X_v
u(C_t^{-1}x)
\\&\qquad + O_{\ell+2+\ell(Z)}(t^{\ell+2}, u, C_t^{-1} x).                 \end{aligned}
\end{equation}
  \noindent{\emph{Cancellation property.}}  If $s\ge 4$,    let $\ell\ge 2$ be such that   $\ell+2\le s$.
    If formulae $F_{\ell,1}$ \dots, $F_{\ell, \ell+2}$ hold, then,
    for any  $C_t(X_{w_1}, \dots, X_{w_\ell})$ and    $Z$ such that $\ell+2+\ell(Z)\le s+1$,
    the coefficients $c_w ,\wt c_w
    $ in \eqref{brandi} satisfy
\begin{equation}
 \label{lls}
 \sum_{|w|=\ell+\ell(Z)+1}(c_w+\wt c_w)X_w(x)=0\qquad \text{for any }\;
x\in K.
\end{equation}
  \noindent{\rm Formulae  $F_{\ell,r}$, with $\ell+3\le r\le s$.} Let $s\ge 5$ and
   assume that $\ell\ge 2$ and $r$ are such that
  $ \ell+3\le r\le    s$. Then,
   for any  $C_t(X_{w_1}, \dots, X_{w_\ell})$ and    $Z$ with $r+\ell(Z)\le s+1$,
   there are $c_v,\wt c_v$ such that
\begin{equation*}
 \begin{aligned}
                Z(uC_t)(x)-Zu(C_tx)& =
                  t^\ell[Z, X_w]u(C_t x) +
                  \hspace{-2ex}\sum_{|v|=\ell+\ell(Z)+1}^{r-1+\ell(Z)}\hspace{-2ex} t^{|v|-\ell(Z)}c_v X_v
                  u(C_tx)
\\&\qquad + O_{r+\ell(Z)}(t^r, u, C_t x),
\\            Z(uC_t^{-1})(x)-Zu(C_t^{-1}x) & =
-t^\ell[Z, X_w]u(C_t^{-1} x) +\hspace{-2ex}\sum_{|v|=\ell+\ell(Z)+1}^{r-1+\ell(Z)}\hspace{-2ex}
t^{|v|-\ell(Z)} \wt c_vX_vu(C_t^{-1}x)
\\&\qquad + O_{r+\ell(Z)}(t^{r}, u, C_t^{-1} x).                \end{aligned}
\end{equation*}
 \end{theorem}
 \noindent Observe again  that in the
  formula $F_{\ell,k},$ $\ell$ is the length of the commutator which defines $C_t,$ while $k$ is the degree of the power of
  $t$ which controls the remainder.

\begin{proof}[Proof of Theorem \ref{grossetto}. ]
If   $\ell=2$, we have already proved the statement.
See  formulae $(F_{2, 1})$,
$(F_{2, 2})$, $(F_{2, 3})$ and  $(F_{2, 4})$, p.~\pageref{aggiorno},  and recall property \eqref{iii} of the remainders. The
proof will be accomplished in two steps.

\smallskip\noindent{\it{Step 1.}  }
Let $s\ge 4$ and $\ell\ge 2$ be such that $\ell+2\le s$. Assume that   $F_{\ell,1}, F_{\ell,2},\dots, F_{\ell, \ell+2}$ hold. Then the cancellation \eqref{lls} holds for any $C_t(X_{j_1}, \dots, X_{j_\ell})$ and $W$ such that $\ell+2+\ell(W)\le s+1$.

\smallskip\noindent{\it{Step 2.} } Assume that for some $\ell\ge 2$, all formulae $F_{\ell,k}$ hold, for $k=1, \dots, s$.  Then formula $F_{\ell+1,k}$ holds, for any $k=1, \dots, s$.

\smallskip\noindent\it{Proof of Step 1. }\rm Let   $C_t= C_t(X_{w_1}, \dots, X_{w_\ell})$ and $Z$ such that $\ell(Z)+\ell+2\le s+1$.
 Applying twice formula
$F_{\ell,\ell+2}$, we obtain,
\begin{equation}\label{diciotto}
\begin{aligned}
Z u(x)&= Z(uC_t^{-1}C_t)( x)
\\&= Z(uC_t^{-1})(C_t x)
   + t^\ell[Z, X_{w}](uC_t^{-1})(C_t x)
\\& \quad +t^{\ell+1}
\sum_{|v|=\ell+\ell(Z)+1}c_v X_v
(uC_t^{-1})(C_t x)
\\&\quad + O_{\ell+2+\ell(Z)}(t^{\ell+2},uC_t^{-1},C_t x)
\\&
= Zu(x)-t^{\ell}[Z, X_{w}]u(x) +t^{\ell+1}\sum_{|v|=\ell+\ell(Z)+1}\wt c_vX_{v}u(x)
\\&\qquad + O_{\ell+\ell(Z)+2}(t^{\ell+2},u,x)
  + t^\ell [Z, X_{w}](uC_t^{-1})(C_t x)
\\&\qquad + t^{\ell+1} \sum_{|v|=\ell+\ell(Z)+1}c_vX_v(uC_t^{-1})(C_tx)
\\&\qquad + O_{\ell+2+\ell(Z)}(t^{\ell+2}, uC_t^{-1}, C_t x) .\end{aligned}
 \end{equation}
Observe first that property \eqref{note2} gives
\[
 O_{\ell+2+\ell(Z)}(t^{\ell+2}, uC_t^{-1}, C_t x)=
 O_{\ell+2+\ell(Z)}(t^{\ell+2}, u , x).
\]  Later on,   we will tacitly use such
property many times. Recall that $\ell\ge 2. $
By means of $F_{\ell,2}$ and $F_{\ell,1}$, respectively, we obtain
\[\begin{aligned}{}
 [Z, X_{w}](uC_t^{-1})(C_t x) & =[Z, X_{w}]u(x)+O_{2+\ell(Z)+\ell}(t^2,u,x) \quad \text{ and}
\\
 X_v(uC_t^{-1})(C_tx)& = X_vu(x) + O_{2+\ell+\ell(Z)}(t,u,x).  \end{aligned}
\]
Inserting this information into \eqref{diciotto} gives, after algebraic simplifications
\[
\begin{aligned}
0&= t^{\ell+1}  \sum_{|v|=\ell+\ell(Z)+1}\wt c_vX_vu(x) + O_{\ell+2+\ell(Z)}(t^{\ell+2}, u, x) +t^\ell
 O_{2+\ell+\ell(Z)}(t^{2}, u, x)
 \\&\quad + t^{\ell+1}\Big(\sum_{|v|=\ell+\ell(Z)+1}c_vX_vu(x) + O_{\ell+2+\ell(Z)} (t,u,x)\Big).
\end{aligned}
\]
To conclude the proof, recall \eqref{laprop}, divide by $t^{\ell+1}$ and let $t\to 0$.

\medskip\noindent\it Proof of Step 2. \rm
 Let $\ell+2\le s$.
  We prove
  formula
   $F_{\ell+1,\ell+2}$, which is the most significant among all formulae $F_{\ell+1,1}, \dots, F_{\ell+1,s}$. Indeed, once
   $F_{\ell+1,\ell+2}$ is proved, if $\ell+3\le s$, then
   formulae $F_{\ell+1,\ell
   +3}, \dots, F_{\ell+1, s}$  follow easily from $F_{\ell+1, \ell+2}$
    and from property \eqref{iii}. On the other side, the
   lower order formulae
   $F_{\ell+1,k}$ with $k <\ell+2$ are easier (just truncate at the correct
order all the expansions in the proof below).

To start, recall that we are assuming that   $F_{\ell,1}, \dots,
F_{\ell, s}$ hold.
Let, for $t>0$
\begin{equation}\label{ghj}
\begin{aligned}
 C_t:&= C_t(X_{w_1}, \dots, X_{w_\ell})\quad \text{and} \quad
\\C_t^0:&= C_t(X,
 X_{w_1},\dots, X_{w_\ell}) =C_t^{-1} e^{-tX}C_te^{tX},
\end{aligned}
\end{equation}
where  $X=X_{w_0}$. Let $Z$  be a commutator with $\ell(Z)+ \ell+2\le s+1$.
In the subsequent formulae, we expand everywhere up to a remainder of the form
$O_{\ell+2+\ell(Z)}(t^{\ell+2},u,C_t^0 x)$.
By \eqref{skg},
\begin{equation*}
 \begin{aligned}
 Z(u C_t^0) (x) & = Z(uC_t^{-1} e^{-tX}C_t)(e^{tX}x) + t[-X, Z]
 (uC_t^{-1} e^{-tX}C_t)(e^{tX}x)
 \\&\qquad +\sum_{k=2}^{\ell+1} \frac{t^k}{k!}\ad_{-X}^k Z
 (uC_t^{-1} e^{-tX}C_t)(e^{tX}x)
\\&\qquad +O_{\ell+2+\ell(Z)}(t^{\ell+2},uC_t^{-1} e^{-tX} C_t, e^{tX} x)
 \\& =:  (A) +(B) + (C) +O_{\ell+2+\ell(Z)}(t^{\ell+2},u, C_t^0 x),\end{aligned}
\end{equation*}
 where we also used  \eqref{note2}.
Next we use   $F_{\ell,\ell+2}$ in $(A).$
\[
 \begin{aligned}
  (A)&= Z(uC_t^{-1} e^{-tX})(C_te^{tX}x) +t^\ell[Z, X_{w}]
  (uC_t^{-1} e^{-tX})(C_te^{tX}x)
\\&\quad
   + t^{\ell+1}\sum_{|v|=\ell+\ell(Z)+1} c_v X_v
  (uC_t^{-1} e^{-tX})(C_te^{tX}x) +O_{\ell+2+\ell(Z)}(t^{\ell+2},u,C_t^0 x)
  \\&
  =:(A_1)+ (A_2) + (A_3) +O_{\ell+2+\ell(Z)}(t^{\ell+2},u,C_t^0 x).
 \end{aligned}
\]
We first treat $(A_1)$. By  \eqref{skg},
\begin{equation}\label{china}
 \begin{aligned}
  (A_1)&=Z(uC_t^{-1} )(e^{-tX}C_te^{tX}x) +t[X,Z](uC_t^{-1} )(e^{-tX}C_te^{tX}x)
  \\&
  \quad + \sum_{k=2}^{\ell+1} \frac{t^k}{k!}\ad_X^kZ (uC_t^{-1}
  )(e^{-tX}C_te^{tX}x)+ O_{\ell+2+\ell(Z)} (t^{\ell+2},u,C_t^0x).
\end{aligned}\end{equation}
Consider now  the various terms in $(A_1)$. First use $F_{\ell,\ell+2}$ to get
\[\begin{aligned}
 Z(uC_t^{-1} )(e^{-tX}C_te^{tX}x)
& =Zu(C_t^0 x)
    - t^\ell[Z,X_{w}]u(C_t^0x)
\\&\quad  + t^{\ell+1}\sum_{|v|=\ell+\ell(Z)+1 }\wt c_v
X_{v}u(C_t^0x)
\\& \quad + O_{\ell+2+\ell(Z)}(t^{\ell+2},u,C_t^0 x).\end{aligned}\]
Moreover, by $F_{\ell,\ell+1}$ we get
\[\begin{aligned}\,
  t[X,Z]&(uC_t^{-1} )(e^{-tX}C_te^{tX}x)
  \\& = t\Big\{[X,Z]u( C_t^0 x) -t^{\ell}
  [[X,Z], X_{w}]u(C_t^0x) +O_{\ell+\ell(Z)+2}(t^{\ell+1},u, C_t^0x) \Big\}.
  \end{aligned}
\]
Finally, we use $F_{\ell,\ell+2-k}$ in the $k-$th term of the sum in \eqref{china}. Observe that $\ell+2-k\in \{1, \dots, \ell\}$ so that we use only remainders.
\[
 \frac{t^k}{k!}\ad_X^kZ (uC_t^{-1}
  )(e^{-tX}C_te^{tX}x) =  \frac{t^k}{k!}\Big\{\ad_X^kZ u
  ( C_t^0x) + O_{\ell+\ell(Z)+2}(t^{\ell+2-k},u, C_t^0 x)\Big\}.
\]
Therefore
\[\begin{aligned}(A_1)&
  = Zu(C_t^0 x)
  - t^\ell [Z,X_{w}]u(C_t^0x) + t^{\ell+1}\sum_{|v|=\ell+1+\ell(Z)}\wt c_v
X_{v}u(C_t^0x)
  \\&\quad + t [X,Z]u(C_t^0x)- t^{\ell+1}[[X,Z], X_{w}]u(C_t^0x)
  \\&\qquad +
   \sum_{k=2}^{\ell+1} \frac{t^k}{k!}\ad_X^kZ u
  (C_t^{0}x)+ O_{\ell+2+\ell(Z)}(t^{\ell+2},u,C_t^0 x).
   \end{aligned}
\]

Next we consider $(A_2)$. Formula \eqref{skg} gives
\[
 \begin{aligned}
  (A_2)&= t^\ell[Z, X_w](uC_t^{-1})(e^{-tX}C_t e^{tX} x)
  \\&\qquad  + t^{\ell+1}[X, [Z, X_w]]
(uC_t^{-1})(e^{-tX}C_t e^{tX} x)
+O_{\ell+\ell(Z)+2}(t^{\ell+2}, u, C_t^0 x). \end{aligned}
\]
Since $\ell\ge 2$, formulas $F_{\ell,2} $  and $F_{\ell,1}$ give respectively
\[\begin{aligned}
 t^\ell[Z, X_w](uC_t^{-1})(e^{-tX}C_t e^{tX} x)&
=t^\ell[Z, X_w]u(C_t^0 x)
+ O_{\ell+\ell(Z)+2}(t^{\ell+2}, u, C_t^0 x) ,
\\  t^{\ell+1}[X, [Z, X_w]](uC_t^{-1})(e^{-tX}C_t e^{tX} x)&
= t^{\ell+1}[X, [Z, X_w]]u(C_t^0 x)
\\&\quad  + O_{\ell+\ell(Z)+2}(t^{\ell+2},
   u, C_t^0 x),\end{aligned}
\]
so that
\[
 \begin{aligned}
  (A_2)&= t^\ell[Z, X_{w}]u(C_t^0x) + t^{\ell+1}[X, [Z, X_{w}]]
  u(C_t^0x)+O_{\ell+\ell(Z)+2}(t^{\ell+2},u, C_t^0 x).
 \end{aligned}
\]

To handle $(A_3)$,   observe that  a repeated application of \eqref{skg}  gives
\[
(A_3)=
  t^{\ell+1}\sum_{|v|=\ell+\ell(Z)+1} c_v X_v
  u(C_t^0x)+O_{\ell+\ell(Z)+2}(t^{\ell+2},u,C_t^0 x).
\]

Next we study $(B)$. Start with formula $F_{\ell, \ell+1}$:
\[
 \begin{aligned}
  (B)&=
 t[-X,Z](uC_t^{-1} e^{-tX})(C_te^{tX}x)  + t^{\ell+1}[[-X, Z], X_{w}](uC_t^{-1}
e^{-tX})(C_te^{tX}x)
 \\&\qquad + O_{\ell+\ell(Z)+2}(t^{\ell+2}, u, C_t^0 x)
  \\&= t[-X,Z](uC_t^{-1} e^{-tX})(C_te^{tX}x)
 \\&\qquad
 + t^{\ell+1}[[-X, Z], X_{w}]
  u(C_t^0x)+ O_{\ell+\ell(Z)+2}(t^{\ell+2}, u, C_t^0x)
\\& = : (B_1) + (B_2)+  O_{\ell+\ell(Z)+2}(t^{\ell+2}, u, C_t^0x),
 \end{aligned}
\]
We first consider  $(B_1)$. In view of \eqref{skg}, we obtain
\[
 \begin{aligned}
(B_1)&=   t[-X, Z](uC_t^{-1} )(e^{-tX}C_te^{tX}x)
 - t\sum_{k=1}^\ell \frac{t^k}{k!}\ad_{X}^k [X,
Z](uC_t^{-1})(e^{-tX}C_te^{tX}x
 )\\& +t O_{\ell+\ell(Z)+2}(t^{\ell+1}, u, C_t^0 x).
  \end{aligned}
\]
But by $F_{\ell,\ell+1}$ we get
\[
 \begin{aligned}
  t[-X, Z](uC_t^{-1} )(e^{-tX}C_te^{tX}x)&= t[-X, Z]u(C_t^0
  x)-t^{\ell+1}[[-X,Z], X_w]u(C_t^0 x)
  \\&
  +O_{\ell+\ell(Z)+2}(t^{\ell+2},u, C_t^0 x),
 \end{aligned}
\]
while for any $k=1, \dots, \ell$,
 formula $F_{\ell, \ell+1-k}$ gives
 \[\begin{aligned}
      t\frac{t^k}{k!} &\ad_X^k([X, Z])
 (uC_t^{-1} )
 (e^{-tX}C_te^{tX}x)
\\&=
 \frac{t^{k+1}}{k!} \ad_X^{k+1}Z u (C_t^0x)+
 O_{\ell+\ell(Z)+2}(t^{\ell+2 }, u, C_t^0x) .
   \end{aligned}
 \]
 Therefore
 \[
 \begin{aligned}
(B_1)&=t^{\ell+1}[[X,Z],X_{w}]u(C_t^0 x)
  - \sum_{k=0}^\ell
 \frac{t^{k+1}}{k!}\ad_{X}^{k+1} Zu(C_t^0 x)+O_{\ell+\ell(Z)+2}(t^{\ell+2}, u,
C_t^0x).
  \end{aligned}
\]
Observe that   $t^{\ell+1}[[X,Z],X_{w}]u(C_t^0 x)=-(B_2)$.

Finally we consider $(C)$. In the $k-$th term of the sum use   formula
$F_{\ell,\ell+2-k}$. Then
\[
 \begin{aligned}
  (C)&=\sum_{k=2}^{\ell+1}
  \frac{t^k}{k!} \ad_{-X}^kZ (uC_t^{-1}e^{-tX})(C_t e^{tX} x)
  +O_{\ell+2+\ell(Z)}(t^{\ell+2},u, C_t^0x) =\text{by (\ref{skg})}
 \\&=\sum_{k=2}^{\ell+1}\frac{t^k}{k!}\Big\{\sum_{h=0}^{\ell+1-k}\frac{t^h}{h!}
  \ad_{X}^h\ad_{-X}^k Z(uC_t^{-1})
  (e^{-tX}C_te^{tX}x)
  \\&\qquad\qquad+O_{\ell+\ell(Z)+2}(t^{\ell+2-k},u,C_t^0x)
  \Big\}+O_{\ell+\ell(Z)+2}(t^{\ell+2},u, C_t^0 x)
  \\&=
 \sum_{k=2}^{\ell+1} \sum_{h=0}^{\ell+1-k}\frac{t^{k+h}}{k!h!}
  (-1)^k\ad_{X}^{k+h} Z u
  (C_t^{0} x)
  +O_{\ell+\ell(Z)+2}(t^{\ell+2},u, C_t^0 x).
 \end{aligned}
\]

Collecting together all the previous computations and making some
simplifications (in particular we need  here the cancellation
property \eqref{lls}),  we get
\[
 \begin{aligned}
  Z(uC_t^0 )(x)&= (A_1)+ (A_2)+(A_3)+(B_1)+(B_2)+(C)
  \\& =Zu(C_t^0 x)
 + t^{\ell + 1}\big\{-[[X, Z],X_{w}]
  u(C_t^0 x)
  + [X,[Z, X_{w}]] u(C_t^0 x)\big\}
\\&\quad+\sum_{k=1}^{\ell+1}\frac{t^k}{k!}\ad_X^kZu(C_t^0 x)
-\sum_{k=0}^{\ell}\frac{t^{k+1}}{k!}\ad_X^{k+1}Zu(C_t^0 x)
\\&
\quad+ \sum_{k=2}^{\ell+1} \sum_{h=0}^{\ell+1-k}\frac{t^{k+h}}{k!h!}
  (-1)^k\ad_{X}^{k+h}Zu(C_t^0 x)
  +O_{\ell+\ell(Z)+2}(t^{\ell+2},u, C_t^0 x)
\\&=: Zu(C_t^0 x)+t^{\ell+1}\{\cdots\}  +(1)+(2)+(3) +
O_{\ell+\ell(Z)+2}(t^{\ell+2},u, C_t^0 x). \end{aligned}
\]
The Jacobi identity gives
$
t^{\ell+1} \{ \cdots\}=t^{\ell+1} [Z, [X, X_{w}]],
$
which is the desired term.

Ultimately we need to consider all the terms with sums.  Changing  $k$
and $h$ in $(2)$, we may write
\[
\begin{aligned} (2)+(3) & =\sum_{k=1}^{\ell+1}\sum_{h=0}^{\ell+1-k}(-1)^k\frac{t^{k+h}}{k!h!}
\ad_X^ { k+h }
 Zu(C_t^0 x)\qquad\text{ and}
\\
 (1)+Zu(C_t^0 x) & = \sum_{h=0}^{\ell+1}\frac{t^h}{h!}\ad_X^hZ u(C_t^0 x).
\end{aligned}\]
Therefore,
\[
\begin{aligned} (1)+(2) +(3)
  +Zu(C_t^0
  x)&
=\sum_{k=0}^{\ell+1}\sum_{h=0}^{\ell+1-k}(-1)^k\frac{t^{k+h}}{k!h!}\ad_X^{k+h}
 Zu(C_t^0 x)
 \\&= \sum_{s=0}^{\ell+1 }\Big( \sum_{\substack{k+h=s\\k,h\ge 0}}\frac{(-1)^k}{k!h!}
\Big) t^s\ad^s_X Z u(C_t^0 x)
=Zu(C_t^0x),
\end{aligned}\]
because  $\displaystyle{\sum_{\substack{k+h=s\\k,h\ge 0}}\frac{(-1)^k}{k!h!}=0}$ for all $s\ge 1$.
The proof of {\it Step 2 } and of Theorem  \ref{grossetto} is concluded. \end{proof}

\subsection{Derivatives of approximate exponentials}\label{cuil2}

Here we give the formula for the derivative of an approximate exponential. All the subsection is written for
the mollified  vector fields $X_j^\s$, but we drop everywhere the supesrcript.

 \label{tazzurella}

\begin{theorem}\label{poiche}
 There is $t_0>0$   such that, for any
$\ell\in\{2,\dots, s\} $,  $w=(w_1, \dots, w_\ell)$, letting $C_t=C_t(X_{w_1}, \cdots, X_{w_\ell})$, there are constants $a_w, \wt a_w$
 such that,  for any $x\in K$, $t\in
 [0, t_0]$,
\begin{equation}
  \label{G1}
  \frac{d}{dt}u(C_t x )= \ell t^{\ell-1} X_wu(C_tx) +\sum_{|v|=\ell+1}^s a_v t^{|v|-1}
  X_vu(C_tx)
+O_{s+1}(t^s,u, C_tx),
\end{equation}
\begin{equation}  \label{G2}
\begin{aligned}
  \frac{d}{dt}u(C_t^{-1} x) = & - \ell t^{\ell-1} X_w u(C_t^{-1}x) +\sum_{|v|=\ell+1}^s \wt a_v
  t^{|v|-1}
  X_v u(C_t^{-1}x)
\\&  +O_{s+1}(t^s, u,C_t^{-1}x).
\end{aligned}
\end{equation}
where, if $\ell=s$, the sum is empty, while, if   $ 2\le \ell<s$, we
have   the  cancellation
\begin{equation}
 \label{G3}
 \sum_{|w|=\ell+1}\big\{a_w+\wt a_w\big\} X_w(x)=0\quad\text{ for all } x\in K.
\end{equation}
\end{theorem}

From Theorem \ref{poiche} it is  very easy to obtain the following result:

\begin{theorem}\label{dididi}For any commutator $X_w$ with length $|w|=\ell\le s$, we have, for $x\in K$ and $t\in[-t_0, t_0] $,
 \begin{equation}\begin{aligned} \label{daeo}
  \frac{d}{dt}u(e_*^{tX_w}(x))& = X_wu(e_*^{tX_w}(x))
   +\sum_{|v|=\ell+1}^s \a_v(t) X_vu
 \big(e_*^{tX_w}(x)\big)
 \\& +O_{s+1}\big(|t|^{(s+1-\ell)/\ell},u,e_*^{tX_w}(x)\big),
 \end{aligned}\end{equation}
where the sum is empty if $\ell=s$, $\a_v(t) =  \ell^{-1} a_v   t^{(|v|/\ell)-1)}$,
 if $t>0$ and $\a_v(t) =- \ell^{-1} \wt a_v   |t|^{(|v|/\ell)-1)}$ if $t<0$.
In particular, the map $(t,x)\longmapsto e_*^{tX_w}(x)$ is of class
$C^1$ on $(-t_0,t_0)\times\Omega'.$
  \end{theorem}
Example \ref{wright} shows that, even if the vector fields are smooth, then the map $\exp_*(t X_w)$
is at most $C^{1, \a}$ for some $\a< 1$.
\begin{proof} [Proof of Theorem \ref{dididi}] Formula \eqref{daeo} follows immediately from \eqref{G1}, \eqref{G2} and
the definition \eqref{approdo} of $e_*$.
We only need to show now that the map is $C^1$ in both variables $t,x$.

Recall that the vector fields $X_j^\s$ are smooth and in particular $C^1$. By classical ODE
theory, see \cite[Chap.~5]{Ha},
 any map of the form
$(\t_1, \dots, \t_\nu,x)\mapsto  e^{\t_1 X_{i_1}}\cdots e^{\t_\nu
X_{i_\nu}}x$ is $C^1$ if the $\t_j$'s belong to some neighborhood of the origin
and $x\in \Omega'$.
This implies that for any commutator $X_w$, the
map $\nabla_x \exp_*(tX_w)x$ is   continuous on
$(t,x)\in I \times \Omega'$, while
$\frac{d}{dt} \exp_*(tX_w)x $ is continuous in
$(t,x)\in \big(I\setminus\{0\}\big)\times \Omega'$.

Next we prove that  $\frac{d}{dt}\exp_*(tX_w)x$ exists and it is
continuous also at all points of the form $(0,x)$. Observe first
that formula \eqref{daeo} gives
\begin{equation}
\label{cjc}
\lim_{t\to 0}\frac{d}{d t}\exp_*(tX_w)x=  X_w(x)\qquad\text{uniformly in } x\in
\Omega'.
\end{equation}
Now, \eqref{cjc} and l'H\^opital's rule imply that
 $
 \frac{d}{dt} \exp_*(tX_w)x\big|_{t=0} = X_w(x),
 $ for all $x\in \Omega'$. Finally, the uniformity of the limit  ensures that
the map
 $(t,x)\mapsto\frac{d}{dt} \exp_*(tX_w)x $ is actually continuous in
$I\times\Omega'$. \end{proof}

\begin{proof} [Proof of  Theorem \ref{poiche}.] We divide the proof in two steps.

\smallskip\noindent {\it Step 1. } We first prove that, if \eqref{G1} and
\eqref{G2} hold for some $w$ with  $\ell:= |w|\in \{2,\dots,s-1\}$,
then the cancellation  formula \eqref{G3} must hold. Fix such a $w$
and start from the identity $\frac{d}{dt}u(C_t^{-1}C_t x)=0$.
\begin{equation}\begin{aligned}
0 & =\frac{d}{ds}u(C_s^{-1}C_t x)\Big|_{s=t}+\frac{d}{ds}(uC_t^{-1})(C_s x)\Big|_{s=t}
 \label{quazi}
\\&= -\ell t^{\ell-1} X_wu(x)   +\sum_{|v|=\ell+1}\wt a_v t^\ell X_vu(x) +
O_{\ell+2}(t^{\ell+1},u,x)
+
\ell t^{\ell-1} X_w(uC_t^{-1})(C_tx) 
\\&\quad \quad
 +\sum_{|v|=\ell+1} a_v t^\ell X_v(uC_t^{-1})(C_tx) +
O_{\ell+2}(t^{\ell+1},uC_t^{-1},C_tx).               \end{aligned}
\end{equation}
But, since $\ell\ge 2$,  formula  $F_{\ell,2}$ shows that
\[\begin{aligned}
 t^{\ell-1}\big\{X_w(uC_t^{-1})(C_tx)
 -X_wu(x) \big\}&=  t^{\ell-1}O_{2+|w|}(t^2,u,x)=O_{\ell+2}(t^{\ell+1},u,x),
 \end{aligned}
\]
while $F_{\ell,1}$ gives for any $v$ with $|v|=\ell+1$,
\[
 t^\ell\{X_v(uC_t^{-1})(C_tx)-X_v u(x)\}=t^\ell O_{1+|v|}(t,u,x)=
O_{\ell+2}(t^{\ell+1},u,x).
\]
Divide \eqref{quazi} by $t^\ell$ and let $t\to 0$ to get
\eqref{G3}. Step 1 is concluded.

\medskip\noindent{\it{Step 2.}  } We prove by an induction argument, that, if
Theorem \ref{poiche} holds for some $\ell\in \{2, \dots, s-1\} $, then it holds
for $\ell+1$. To show the result for $\ell=2$, it suffices to follow the proof below, taking into account that formulas \eqref{G1} and \eqref{G2} are trivial, if $\ell=1$.  We use the notation in \eqref{ghj} for $C_t$ and $C_t^0$.
In view of \eqref{iii} and of the   already accomplished  Step 1, it
suffices  to prove that
\begin{equation}\label{lsd}\begin{aligned}
\frac{d}{dt} u(C_t^0x) &=(\ell+1)t^{\ell}[X,X_w]u(C_t^0x)+O_{\ell+2}(t^{\ell+1},
u, C_t^0 x)\quad\text{and}
\\
\frac{d}{dt}
u((C_t^0)^{-1}x)
&=-(\ell+1)t^{\ell}[X,X_w]u((C_t^0)^{-1}x)+O_{\ell+2}(t^{\ell+1}, u,
(C_t^0)^{-1} x).
                           \end{aligned}
\end{equation}
We prove only the first line of  \eqref{lsd}.  The  latter  is similar.
We know that
\begin{equation}\label{derivatives}
\begin{split}
\frac{d}{dt}\left(u(C_tx)\right)=&\ell
t^{\ell-1}X_wu(C_tx)+t^\ell\sum_{|v|=\ell+1}
a_v X_vu(C_tx)+O_{\ell+2}(t^{\ell+1}, u, C_tx)\\
\frac{d}{dt}u(C^{-1}_tx) =
&-\ell t^{\ell-1}X_wu(C^{-1}_tx)+t^\ell\sum_{|v|=\ell+1}\wt
a_vX_vu(C^{-1}_tx)
\\&\quad +O_{\ell+2}(t^{\ell+1}, u, C_t^{-1} x),
\end{split}
\end{equation}
with  the remarkable cancellation \eqref{G3}. Observe that $a_v=\wt a_v=0$,  if $\ell=1$.
Next,
\[
\begin{split}
\frac{d}{dt}\left(u(C^0_tx)\right)=
&\frac{d}{dt}\left(u(C^{-1}_te^{-tX}C_te^{tX}x)\right)
\\
=
& X\left(uC^{-1}_te^{-tX}C_t\right)(e^{tX}x)
+
\frac{d}{ds}(uC_t^{-1} e^{-tX})(C_s e^{tX} x)\Big|_{s=t}
\\
&- X\left(uC^{-1}_t\right)(e^{-tX}C_te^{tX}x)
+
\frac{d}{ds}u(C_s^{-1} e^{-tX} C_t e^{tX} x)\Big|_{s=t}
\\
=: &A_1+A_2+A_3+A_4.
\end{split}
\]

First   we study $A_1+A_3$, by \eqref{skg} and
  $F_{\ell,\ell+1}$.
\[
\begin{split}
A_1+A_3 =& X\left(uC^{-1}_te^{-tX}C_t\right)(e^{tX}x)
-X\left(uC^{-1}_te^{-tX}\right)(C_te^{tX}x)
= \text{ (by $F_{\ell,\ell+1}$)}
\\
=& t^\ell [X,X_w](uC^{-1}_te^{-tX})(C_te^{tX}x)+O_{\ell+2}(t^{\ell+1}, uC_t^{-1}e^{tX}, C_t e^{tX}x)
\\= & t^\ell  [X,X_w] u(C^0_tx)+O_{\ell+2}(t^{\ell+1}, u, C_t^0 x).
\end{split}
\]

Next we study $A_2+A_4$, by means of \eqref{derivatives}.
\[
\begin{split}
A_2+A_4=&
\ell
t^{\ell-1}X_w(uC^{-1}_te^{-tX})(C_te^{tX}x)+t^\ell
\sum_{|v|=\ell+1}a_vX_v(uC^{-1} _te^ { -tX})(C_te^{tX}x)
\\
&- \ell t^{\ell-1}X_wu(C^0_tx)+t^\ell\sum_{|v|=\ell+1}\tilde a_vX_vu(C^0_tx)
+O_{\ell+2}(t^{\ell+1}, u, C_t^0 x)
\\
=&
\ell t^{\ell-1}\Big\{
X_w(uC^{-1}_t)(e^{-tX}C_te^{tX}x)+t[X,X_w](uC^{-1}_t)(e^{-tX}C_te^{tX}x)
\\&\qquad \qquad +O_{\ell+2}
(t^2, u ,  C_t^0 x)
\Big\}
\\
&+t^\ell\sum_{|v|=\ell+1}a_v\big\{ X_v u(C^0_tx) + O_{\ell+2}(t,u,C_t^0x)\big\}
\\
&- \ell t^{\ell-1}X_wu(C^0_tx)+t^\ell\sum_{|v|=\ell+1}\tilde a_vX_vu(C^0_tx)
+O_{\ell+2}(t^{\ell+1}, u, C_t^0 x).
\end{split}
\]
Now observe that by formula $F_{\ell,2}$ we have, if $\ell\ge 2$,
\[
 \begin{aligned}
  t^{\ell -1}&\Big\{X_w(uC_t^{-1})(e^{-tX}C_t e^{tX}x)-X_wu(C_t^0x) \Big\}
 =O_{\ell+2} (t^{\ell+1}, u, C_t^0 x),
\end{aligned}
\]
while, if $\ell=1$ the left-hand side vanishes identically. Thus,     cancellation \eqref{G3} gives
$A_2+A_4 =
\ell t^{\ell}[X,X_w]u(C^0_tx) +O_{\ell+2}(t^{\ell+1}, u, C_t^0 x)
$ and ultimately
$
A_1+A_2+A_3+A_4=(\ell+1)t^{\ell}[X,X_w]u(C^0_tx)+O_{\ell+2}(t^{\ell+1},u, C_t^0
x).
$ The proof is concluded.
\end{proof}

  \section{Persistence of maximality conditions on balls}\setcounter{equation}{0}
\label{jkt}
Here we establish a key property of stability of the $\eta-$maximality condition.  The argument, as in \cite{TW},
is based on Gronwall's   inequality.
 \label{duesette}
\begin{theorem}\label{dapro}
Let $X_1,\dots, X_m$ be vector fields  in $\A_s$.
Then, there are   $r_0>0$ and  $\e_0>0$ depending on the constants $L$  and
$\nu$
in \eqref{horma} and \eqref{lipo} such that, if for some $\eta\in
\left]0,1\right[,$ $x\in K$ and $r<r_0$, the triple
$(\I,x,r) $ is  $\eta-$maximal,
 then for any $y\in B(x,\eta \e_0
r)$, we have the estimates
\begin{equation}\label{aceto}
|\lambda_\I(y)-\lambda_\I(x)|\le \frac{1}{2}|\lambda_\I(x)|,
\end{equation}
\begin{equation}\label{tredici}
|  \lambda_\I(y)|r^{\ell(\I)}> C^{-1}\eta \Lambda(y,r).
 \end{equation}
 \end{theorem}

To prove Theorem \ref{dapro}  we need the following easy lemma.

\begin{lemma}\label{rough}
There is $C>0$ depending on $L$ and $\nu$ such that, given $y\in
\Omega $ and $z\in\R^n$, the linear system
$
 \sum_{i=1}^q {Y_i(y)}\xi^i = z
$
 has a solution $\xi\in\R^q$ such that
$
 |\xi|\le C |z|
$.
\end{lemma}
\begin{proof}    Take       $y\in
\Omega$ and choose $(k_1,\dots, k_n)\in\ii$ such that
$\det(Y_{k_1}(y),\dots ,  Y_{k_n}(y))\ge \nu$.  Let
$
A := (Y_{k_1}(y),\dots, Y_{k_n}(y))
$.  Thus,
$
 \abs{A^{-1}}\le  C / |\det(A)|  \le C/\nu,
$ where $C$ depends on $L$.
The lemma is easily proved by studying the system
$A\xi=z
$ with $\xi= (\xi_{k_1}, \dots, \xi_{k_n})\in\R^n$. \end{proof}

\begin{proof}[Proof of Theorem \ref{dapro}]
Observe that if $(I, x, r)$ is $\eta$-maximal, then
 there is $\wt\sigma>0$ which may also depend  on $I, x, r$,
  such that $(I, x, r)$ is $\eta$-maximal for the mollified $X_j^\s$ for all $\s\le \wt \sigma$.
Therefore, we will give the proof for smooth vector fields (without writing any superscript).
The nonsmooth case will  follow by  passing to the limit as $\sigma\to 0$ and taking into account that all
constants are stable.

  Let $J\in \ii $ and let
$ \lambda_J(x): = \det[Y_{j_1} (x), \dots, Y_{{j_n}} (x)] $.
Let $X$ be a vector field of length one.
Recall the following formula  (see \cite[Lemma 2.6]{NSW}):
\begin{equation*}
 \begin{aligned}
X\lambda_\J &
=   (\div X) \lambda_\J
   +\sum_{k=1}^n \det( \dots,Y_{j_{k-1}},
  [X, Y_{j_k}], Y_{j_{k+1}}, \dots )
 \\&
=  (\div X) \lambda_\J
   +\sum_{k\le n ,\;     \ell_{j_k}<s} \det( \dots,Y_{j_{k-1}},
  [X, Y_{j_k}], Y_{j_{k+1}}, \dots )
  \\ &\qquad  +\sum_{k\le n ,\;\ell_{j_k}=s} \det( \dots,Y_{j_{k-1}},
  [X, Y_{j_k}], Y_{j_{k+1}}, \dots )
  \\&
  =:(A) +\sum_{k\le n ,\;     \ell_{j_k}<s}  (B)_k
+\sum_{k\le n ,\;\ell_{j_k}=s}(C)_k   .
 \end{aligned}
\end{equation*}
We claim that
\begin{equation}\label{abcpp}
r^{\ell(\J)}\, |X\lambda_J(y)|\le \frac{C}{r}\Lambda(y,r)   \qquad \text{for all}
 \quad y\in \Omega  \quad \J\in \ii \quad r\le r_0.
\end{equation}

To prove \eqref{abcpp},
 observe first that, if $y\in\Omega$, then
\[|(A) (y)|\le C   |\lambda_\J (y)|
   \le C r^{-\ell(\J)}  \Lambda(y,r)
,\] by definition of $\Lambda$. This  gives immediately the correct estimate for $(A)$.
Next we look at $(B)_k $.  Since $\ell(Y_{j_k})\le s-1  $,
we get for $y\in\Omega$
\[\begin{aligned}
|(B)_k (y)| & \le   \big|\det\big(  \dots, Y_{j_{k-1}}, [X, Y_{j_k}], Y_{j_{k+1}},\dots\big) (y)\big|
\le C\Lambda(y,r)\, r^{-\ell(\J)-1}  .
\end{aligned}\]
Finally we consider   $(C)_k $. In view of   Lemma \ref{rough}, we may  write
$[X, Y_{j_k} ](y)=\sum_{i=1}^q\xi_{j_k,i}  Y_i(y)$, where $|\xi_{j_k,i} |\le C\big|[X, Y_{j_k}](y)\big| \le C$
for any $i=1, \dots, q$.
Therefore,
\[
 \begin{aligned}
|(C)_k(y)| &
=
\Big| \sum_{i=1}^q \xi_{  j_k,i}
\det\big(  \dots, Y_{j_{k-1}}, Y_i, Y_{j_{k+1}},\dots\big) (y)\Big|
\\&\le
C\sum_{i=1}^q
\Big| \det\big(  \dots, Y_{j_{k-1}}, Y_i, Y_{j_{k+1}},\dots\big) (y)\Big|
\\&
\le C  \sum_{i=1}^q  r^{-\ell(\J)+ \ell_{j_k}-\ell_i}\Lambda(y, r)
\le   C\Lambda(y, r)r^{-\ell(\J)-1},
\end{aligned}
\]
because  $\ell_{j_k}=s\ge \ell_i $, for any $i=1,\dots,q$. This finishes the proof of \eqref{abcpp}.

Let  $\gamma:[0, r]\to \R^n$
 be a subunit path with $\gamma(0)=x$,
$\gamma(r)=y$. Assume that $x\in K$ and $r$ is small enough to ensure that $B(x, r)\subset\Omega $.
 Then, by \eqref{abcpp}
\begin{equation}\label{zoppos}
\begin{aligned}
r^{\ell(\J)}\, | \lambda_{\J} (y) - \lambda_{\J} (x)| & \le \frac{C}{r}\int_0^r \Lambda(\g(s), r)ds.
\end{aligned}
\end{equation}

To have differentiability, define
$
 \Lambda_2(x, r):=\Big\{\sum_{\I}\Big( \lambda_\I(x) \,r^{\ell(I)}\Big)^2\Big\}^{1/2},
$
which is equivalent to $\Lambda(x,r)$, through absolute constants.
Therefore, \eqref{abcpp} gives
\begin{equation*}
\begin{aligned}
 \left|\frac{d}{d s} \Lambda_2  ( \gamma(s),r)\right|
 & =\Big|
 \frac{1}{\L_2 (\gamma(s),r)}
\sum_{\J} r^{2\ell(\J)}\lambda_\J (\gamma(s))
 \p_s\lambda_{\J} (\g(s))\Big|   \le  \frac{C}{r}\Lambda_2(\gamma(s), r) .
\end{aligned}
\end{equation*}
 Integrating the inequality we get
\begin{equation}\label{backspace}
\big|\L_{2}(x,r)- \L_{2}(\gamma(s),r)\big|\le \L_{2}(x,r)
  \Big(\exp\Big( \frac{C}{r}s \Big) -1 \Big). \end{equation}
Moreover, integrating \eqref{zoppos} for $\J=\I$,  we get for $0\le s\le r$,
\begin{equation}\label{nikon}
\begin{aligned}
 \big| r^{\ell(\I)}\lambda_\I(\gamma(s))
  -  r^{\ell(\I)}\lambda_\I(x)\big|
&\le \frac{C}{r}\int_0^s\Lambda_2(\gamma(\t),r)d\t
\\&\le \frac{C}{r} \int_0^s\Lambda_2(x,r)e^{C\tau/r}d\tau =\Lambda_2(x,r) \big(e^{C s/r} -1\big)
\\&\le \frac{Cs}{r}\Lambda(x,r)\le \frac{Cs}{r\eta}r^{\ell(\I)} \,|\lambda_\I(x)|,\end{aligned}
\end{equation}
because $(I,x,r)$ is $\eta$-maximal.
Then \eqref{aceto} and \eqref{tredici} follow from \eqref{nikon} and \eqref{backspace}.
\end{proof}

At this point we can prove the following statement.
\begin{corollary}
 \label{765}Assume that  $(\I,x,r)$ is $\eta-$maximal for the vector fields $X_1, \dots, X_m$ in $\A_s$ and  for some $x\in K$ and  $r\le r_0$.
Then for any $y\in B(x,\e_0 \eta  r)$, $i=1,\dots,q$,  we may   write $Y_j(y)=\sum_{k=1}^n a_j^k Y_{i_k}(y)$, where
 $
  |a_j^k|\le \frac{C}{\eta}  r^{\ell_{i_k} -\ell_j}.
 $
\end{corollary}
\begin{proof} Write $Y_i$ instead of $Y_i(y)$. Look at the linear system $Y_j=\sum_{k=1}^n a_j^k Y_{i_k}. $ The Cramer's rule furnishes
\[
 \begin{aligned}
  a_j^k&= \frac{\det[Y_{i_1}, \dots, Y_{i_{k-1}}, Y_j, Y_{i_{k+1}}, \dots, Y_{i_n} ]}{\det[Y_{i_1}, \dots,
  Y_{i_{k-1}}, Y_{i_k},
  Y_{i_{k+1}}, \dots, Y_{i_n} ]}
  \le \frac{C}{\eta} r^{\ell_{i_k}-\ell_j},
 \end{aligned}
\]
by \eqref{tredici},  and the proof is concluded. \end{proof}

\section{Ball-box theorem  }\label{uno4}
\subsection{Derivatives of  almost exponential maps }
\setcounter{equation}{0}\label{unotre} Here we take H\"ormander
vector fields  $X_1,\dots, X_m$  in $\A_s$.    When we choose an
$n-$tuple $\I=(i_1, \dots, i_n)\in\ii$ and the $n-$tuple is
understood, we write  $Y_{i_j} =U_j  $ and $\ell(Y_{i_j}) =\ell(U_j)
=d_j$, for $j=1,\dots,n$. Our first result is:

\begin{theorem}\label{ema3}  There are $
\s_0, r_0,\s_0$ and $C>0$
  such that, given $\I\in\ii$,   then, for any $j=1, \dots, n$, $\s\le\s_0$,
$x\in K$ and   $h\in Q_\I(r_0)$,
the $C^1$ map $E_{I,x}^\s$ satisfies
\begin{equation}\label{dcd3}
\begin{aligned}
 \frac{\p}{\p h_j} E^\s_{\I,x}(h) & = U_j^\s  (E^\s_{\I,x}(h)) +
 \sum_{|w|=d_j +1}^{s}a^w_j(h)X_{w}^\s (E^\s_{\I,x}(h)) + \omega^\s_j(x,h),
\end{aligned}
\end{equation}
where the sum is empty if $d_j=\ell(U_j)=s$ and the following estimates hold:
\begin{equation}\label{wiz3}
 |\o_j^\s(x, h)|\le C \|h\|_\I^{s+1-d_j } \qquad\text{for any $x\in K$}\quad   h\in Q_\I(r_0)\quad \s\le\s_0,
\end{equation}
\begin{equation}
 \label{pso3}
|a^w_j(h)|\le C \|h\|_\I^{|w|-d_j } \qquad\text{for all
$  h\in Q_\I(r_0)$} \qquad   |w|=d_j+1,\dots,s.
\end{equation}
 \end{theorem}
  Theorem \ref{ema3} holds without assuming $\eta$-maximality.
If the  triple $(\I,x,r)$  is $\eta-$maximal, we have more.
 To state the result, fix once for all  a dimensional constant $\chi >0$ such that
\begin{equation}\label{aduo}
 \det(I_n+A)\in\Big[ \frac 12,2\Big]
 \qquad  \text{for all } A\in \R^{n\times n} \quad \text{with  norm }\abs{A}\le
\chi  .
\end{equation}

   \begin{theorem}
\label{papero3}
Let $r_0,\s_0>0$ as in Theorem \ref{ema3}. Given an $\eta-$maximal triple
$(\I,x,r)$ for the vector fields $ X_i   $,
with $x\in K$,  $r<r_0$ and $\s\le\s_0$, then, for any $h\in Q_\I(\e_0\eta   r)$, $j=1,\dots,n$, we may write
 \begin{equation}
\begin{aligned}\frac{\p}{\p h_j} E^\s_{\I,x}  ( h) & =  U^\s_j   (E^\s_{\I,x}  ( h)) +
 \sum_{k=1}^n  (b_{j}^k)^\s(x,h)U^\s_k(E^\s_{\I,x}  ( h)),
\end{aligned}
\label{gatto}
\end{equation}
where,
\begin{equation}\label{esso}
 |(b_j^k)^\s(x,h)|\le \frac{C}{\eta} \,\frac{\|h\|_\I}{r} r^{d_k-d_j }
 \le\chi r^{d_k-d_j}\quad\text{for all } \;  h\in Q_\I(\eta\e_0    r).
\end{equation}
\end{theorem}
\begin{remark}\label{pesso} Estimate \eqref{esso} and the results on Section \ref{jkt} imply
 that, under the hypotheses of Theorem \ref{papero3}, we have
\[
 |\lambda^\s_I(x)|\le C_1|\lambda^\s_I (E^\s_{I,x}(h))|\le C_2\left| \det\frac{\p}{\p h}E^\s_{I, x}(h) \right| \le C_3
 |\lambda^\s_I(x)|\quad\text{for all } h\in Q_\I(\e_0\eta r).
\]
\end{remark}

\begin{proof} [Proof of Theorem \ref{ema3}] Without loss of generality we may
work in  $\R^2$.   We drop everywhere the superscript $\sigma$. Then
 $
 E_I(x,h) =e_*^{h_1 U_1}e_*^{h_2 U_2}x.
 $ Denote by $u$ the identity function in $\R^n$.

 We first look at  $\p/\p h_1$.  Theorem  \ref{dididi} with    $X_w=U_1$ and $t=h_1$  gives:
\begin{equation*}
 \begin{aligned}
  \frac{\p}{\p h_1  } u\big(e_*^{h_1 U_1}e_*^{h_2 U_2}x\big) &= U_1u
(E_{\I,x}(h))+\sum_{|v|= d_1+1}^s
  \a_v(h_1) X_vu (E_{\I,x}(h))
  \\&\qquad + O_{s+1}\big(|h_1|^{(s+1-d_1)/d_1}, u, E_{\I,x}(h)\big),
 \end{aligned}
\end{equation*}
 where   we know that $|\a_v(h_1)| \le C
|h_1|^{(|v|-d_1)/d_1}$ and
 \[
    \big|O_{s+1}\big(|h_1|^{(s+1-d_1)/d_1}, u, E_{\I,x}(h)\big)\big|\le
C|h_1|^{(s+1-d_1)/d_1}.
 \]
Thus, since  $|h_1|^{1/d_1}\le \|h\|_\I$, we have proved \eqref{wiz} and
\eqref{pso} for $j=1$.

 Next
we look at the variable  $ h_2$.  Theorem \ref{dididi} gives
 \begin{equation}\label{duetre}
  \begin{aligned}
   \frac{\p}{\p h_2}   u \big( e_*^{h_1 U_1}e_*^{h_2 U_2}x \big)
   &= U_2 \big(u e_*^{h_1 U_1}
   \big)\big
   ( e_*^{h_2 U_2}x
   \big)
   \\&\quad   +\sum_{|v|=d_2+1}^s\a_v(h_2) X_v
  \big(u e_*^{h_1 U_1}
   \big)\big
   ( e_*^{h_2 U_2}x
   \big)
\\&\quad    +O_{s+1}\big(h_2^{(s+1-d_2)/d_2},
 u e_*^{h_1 U_1}
, e_*^{h_2 U_2}x
   \big),
  \end{aligned}
 \end{equation}
 where we know that  $ |\a_v(h_2)| \le C |h_2|^{(|v|-d_2)/d_2}\le
C\|h\|_\I^{|v|-d_2}$ and
 \[\begin{aligned}
 \Big|O_{s+1}\big(h_2^{(s+1-d_2)/d_2},
 u e_*^{h_1 U_1}
, e_*^{h_2 U_2}x
   \big)\Big|
&\le C
|h_2|^{(s+1-d_2)/d_2}\le C\|h\|_\I^{s+1-d_2}. \end{aligned}\]
Now,  a repeated application of  formula \eqref{skg}  gives
\begin{equation}\label{remi}
 \begin{aligned}
   U_2 &\big(u e_*^{h_1 U_1}
   \big)\big
   ( e_*^{h_2 U_2}x
   \big) = U_2u \big(E_{\I,x}(h)\big)
\\&
      +\sum_{ \a_1+\dots+\a_\nu=1}^{s-d_2}
      C_\a h_1^{(\a_1+\cdots+\a_\nu
       )/d_1}\ad^{\a_1}_{Z_1}\cdots\ad^{\a_\nu}_{Z_\nu}U_2 u
      \big(E_{\I,x}(h)\big)
      \\&
     +
      O_{s+1}\big(h_1^{(s+1-d_2)/d_1}, u,E_{\I,x}(h)\big),
\end{aligned}
\end{equation}
where we denoted briefly $e_*^{h_1U_1}=e^{-h_1^{1/d_1}Z_1}\cdots e^{-h_1^{1/d_1}Z_\nu},$
where $\nu$ is  suitable, $h_1>0$ and   $Z_j\in \{  \pm X_1,\cdots\pm X_m\}$. If
$h_1<0 $ the computation is analogue.

 To conclude the proof it suffices   to write
all the terms   $
  X_v(u e_*^{h_1 U_1})(e_*^{h_2 U_2}x)$ in \eqref{duetre}
   in the form  $X_v u \big(E_{\I,x}(h) \big)  $ plus an appropriate remainder.
 The argument is the same used in equation \eqref{remi} and
we leave it to the reader.
\end{proof}

\begin{proof}[Proof of Theorem \ref{papero3}. ] The  proof   relies on
Corollary \ref{765}. We drop everywhere  the superscript $\sigma$.   If
$(\I,x,r)$ is $\eta-$maximal,  then
 \eqref{tredici}  gives $|\lambda_\I(E_{\I,x}(h))| r^{\ell(\I)}\ge C^{-1}\eta \Lambda(E_{\I,x}(h),
 r)$, as soon as $h\in Q_\I(  \e_0 \eta r)$.

Write briefly
 $E$ instead of $E_{I,x}(h)$ . Looking at the right-hand side of \eqref{dcd3},
 we need to study, for any
 word $w $ of length $|w|=\ell$, with
 $\ell= d_j+1,\dots,s$,  the linear system
$
a_j^w(h)X_{w}(E )=\sum_{k=1}^nb_{j}^{k} U_k(E )
$
and we must show that the solution $b_j^k$ satisfies \eqref{esso},
if $\|h\|_{\I}\le \e_0\eta r$. By Corollary \ref{765} write  $X_w(E) = \sum p_w^k U_k(E)$, where $|p_w^k|\le \frac{C}{\eta}r^{d_k-|w|} $.
Thus
\[
 |b_j^k|=  |a_j^w p_w^k|\le C \|h\|_\I^{|w|-d_j }  \frac{C}{\eta}r^{d_k -|w|}
\le \Big( \frac{\|h\|_\I}{r}\Big)^{|w|- d_j}\frac{C}{\eta} r^{d_k - d_j}.\]
 Here we   also used    \eqref{pso3}.
This gives the estimate of the terms in the sum in \eqref{dcd3}.

 Next we look at the  the remainder $\omega_j$. Fix $j=1,\dots, n$.
 We know that $|\o_j|\le C\|h\|_\I^{s+1-d_j}$ and we want to write $\o_j=\sum_{k}b_j^k U_k(E)$ with estimate
 \eqref{esso}.
It is convenient to multiply  by $r^{d_j}$.   Let   $r^{d_j}\o_j=:\theta \in\R^n$ and $\xi^k = r^{d_j} b_j^k$.
Thus it suffices to show that we can  write
  $\theta=\sum_k \xi^k U_k(E)$, where $\xi^k$ satisfies the estimate
 $
  |\xi^k|\le \frac{C}{\eta} \frac{\|h\| }{r} r^{d_k}.
 $
We know   that
  \[
   |\theta|=|r^{d_j}\o_j|\le C\|h\|_\I^{s+1-d_j}r^{d_j}= C\Big(\frac{\|h\|_\I}{r}\Big)^{s+1-d_j}r^{s+1}.
  \]
To estimate $\xi^k$, we   follow  a two steps  argument:

{ \it  Step 1. } Write, by    Lemma \ref{rough},
 $
  \theta=\sum_{i=1}^q \mu^iY_i(E),
 $
for some $\mu\in\R^q$ satisfying $|\mu|\le C |\theta|\le C \Big(\frac{\|h\| }{r}\Big)^{s+1-d_j}  r^{s+1}$.

 {\it Step 2. } For any $i=1,\dots,q$ write $Y_i(E)=\sum_{k=1}^n \lambda_i^k U_k(E)$.
 This can be done in a unique way and estimate
 $|\lambda_i^k|\le \frac{C}{\eta} r^{d_k-\ell(Y_i)}$ holds, by Corollary \ref{765}.

 Collecting  {\it Step 1 }
and { \it Step 2,  } we conclude that
\[
| \xi^k| =\Big|\sum_{i=1}^q \mu^i\lambda_i^k\Big|\le
C \Big(\frac{\|h\| }{r}\Big)^{s+1-d_j} r^{s+1}\cdot \frac{C}{\eta} r^{d_k-\ell(Y_i)}\le
    \frac{C}{\eta}  \frac{\|h\| }{r}  r^{d_k},
\]
as required.   This ends the proof.
\end{proof}

Next we pass to the limit as $\sigma\to 0$ in both Theorems \ref{ema3} and \ref{papero3}.
\begin{theorem}\label{ema}
 If  $(\I,x,r) $ is $ \eta-$maximal for some $x\in K$, $r\le r_0$, then the map $ E_{\I,x}\big|_{Q_\I(\e_0\eta r)}$ is
locally   biLipschitz in the Euclidean sense and satisfies  for  a.e. $h$,
\begin{equation}\label{dcd}
\begin{aligned}
 \frac{\p}{\p h_j} E _{\I,x} (h) & = U_j   (E _{\I,x} (h)) +
 \sum_{|w|=d_j +1}^{s}a^w_j(h)X_{w}  (E _{\I,x} (h)) + \omega_j(x,h)
\\&= U_j(E_{I,x}( h))+\sum_{k=1}^n b_j^k(x,h) U_k(E_{I,x}( h)),
\end{aligned}
\end{equation}
where the sum is empty if $d_j=\ell(U_j)=s$ and otherwise the following
estimates hold:
\begin{equation}\label{wiz}
 |\o_j(x, h)|\le C \|h\|_\I^{s+1-d_j } \qquad \text{for all }  x\in K \quad   h\in Q_\I(r_0),
\end{equation}
\begin{equation}
 \label{pso}
|a^w_j(h)|\le C \|h\|_\I^{|w|-d_j }  \qquad \text{if} \quad  |w|=d_j+1,\dots,s
 \;\text{ and }\;   h\in Q_\I(r_0),
\end{equation}
and \begin{equation}\label{esso3}
| b_j^k(x,h)|\le \frac{C}{\eta}\frac{\|h\|_I}{r} r^{d_k-d_j}
\le \chi r^{d_k-d_j}\qquad\text{for all $h\in Q_I(\e_0\eta r)$}.\end{equation}

 \end{theorem}
\begin{remark}\label{ciuno}
 If $s\ge 3$, then vector fields of the class $\A_s$ are $C^1$. Then, as
discussed in
the beginning of the proof of Theorem \ref{dididi}, the map $E_{I,x}$ is
actually
$C^1$ smooth. This is not ensured if $s=2$.
\end{remark}

\begin{proof}[Proof of Theorem \ref{ema}]
Look first at the $C^1$ map $E^\s=E^{\s}_{\I,x}$ defined on $Q_\I( r_0)
$.
Denote by $E$ its pointwise limit as $\s\to 0$. By Theorem, \ref{ema3}, the map $E^\s $
satisfies for any $\s<\s_0$, $\|h\|_\I\le r_0$,
\begin{equation}\label{dcdc}
\begin{aligned}
 \frac{\p}{\p h_j} E^\s (h) & =U_j^\s (E^\s(  h)) +
 \sum_{|w|=d_j +1}^{s}a^w_j(h)X_{w}^\s(E^\s(h)) + \omega_j^\s(h),
\end{aligned}
\end{equation}
where $a_j^w$ do not depend on $\s$, while  $  |\o_j^\s( h)|\le C
\|h\|_\I^{s+1-d_j },$  uniformly in $\s\le \s_0$.

Let $E^{\s_k} $ be a sequence weakly converging to $E$ in $W^{1,2}$. Therefore, by \eqref{dcdc},
the remainder $\o_j^{\s_k}$ has a weak limit  in $L^2$. Denote it by $\omega_j$. Standard
 properties of weak convergence ensure that $ |\omega_j(h)|\le C_0
\|h\|_\I^{s+1-d_j }$ for
a.e.~$h$.
Therefore, we have proved the first line of \eqref{dcd} and estimates \eqref{wiz} and \eqref{pso}.
 To prove the second line and \eqref{esso3}, it suffices to repeat the argument
of Theorem \ref{papero3}, taking
 into account that the main ingredient there,
namely  Corollary \ref{765}, holds for nonsmooth vector fields in $\A_s$.

Now we have to prove the local injectivity of $E$.
Let $\s$ be small enough to ensure that $(I,x,r)$ is $\eta$-maximal for the
vector fields $X_j^\s$.
 In view of Theorem \ref{papero3}, we can    write
$
 dE^\s(h)=U^\s(E^\s(h)) (I_n+B^\s(h))
$, where $U^\s =[U_1^\s,\dots, U_n^\s]$,
 and   the entries of the matrix $B$ satisfy  $   |( b_j^k)^\s| \le C r^{d_k-d_j }$, by  \eqref{gatto}.
 Fix now
 $h_0\in Q_\I(\e_0\eta  r) $, where $\e_0\eta$ comes from Theorem
\ref{papero3}. We will show that $E^\s$ is locally one-to-one around
  $h_0$, with a stable coercivity estimate as $\sigma\to 0$.  By Proposition
\ref{schiaccia2} and by the continuity of the
  vector fields
  $ U_j$, we may claim that for any
  $\delta>0$ there is $\r>0$ such that $|U_j^\s(\xi)-U_j^\s(\xi')|<\delta$ as soon as $\xi,\xi'\in K$,  $|\xi-\xi'|<\r $ and
  $\s<\r $.
  Recall also that   $  E^\s $
is Lipschitz continuous,  uniformly in $\sigma$, see  \eqref{dcdc}. Then,
   for any   $\d>0$ there is $\r>0$ such that
    $B_{\rm{Eucl}}(h_0,\r)\subset Q_\I(\e_0\eta  r)$, and,
if     $|h-h_0|\le \r$ and $\s<\r$, then   $|U^\s (E^\s(h))- U^\s
(E^\s(h_0))|\le \delta $.

 Take  $h,h'\in B_{\mathrm{Eucl}}(h_0,\delta)$. By integrating on the path $\gamma(t)= h'+t(h-h')$,  we have
\[
 \begin{aligned}
| E^\s(h)-E^\s(h')| &  = \Big|\int_0^1  U^\s(E^\s(\g)) (I+B^\s(\g))(h-h')dt\Big|
\\&\ge \Big| \int_0^1U^\s(E^\s(h_0))(I+B^\s(\g))(h-h') dt \Big|
\\ &\quad
 -\Big|  \int_0^1\big(U^\s(E^\s(\g))- U^\s(E^\s(h_0))\big) (I+ B^\s(\g))(h-h') dt \Big|.
\end{aligned}\]
To estimate from below the first line recall the easy inequality $|Ax|\ge
C^{-1}\frac{|\det{A}|}{|A|^{n-1}} |x|$, for all   $A\in \R^{n\times
n}$.
The pointwise estimate  $   |( b_j^k)^\s| \le \chi  r^{d_k-d_j }$ gives  $|
\int_0^1  (b_j^k)^\s(\g))dt|\le \chi r^{d_k-d_j}$. Thus \eqref{aduo} gives
\[
 \Big|\det \int_0^1(I+ B^\s(\g)) dt \Big|  = \Big|\det \Big(I+\int_0^1
B^\s(\g)dt\Big)\Big|\ge \frac 12.
\]
  Observe also that $\abs{I+B^\s(\g    )}\le Cr^{1-s}$. Moreover, in view of
Remark  \ref{pesso}, it must be $\abs{\det U^\s(E^\s(h_0))}\ge C^{-1}
|\lambda_\I(x)|$, for small $\s$.
 This suffices to estimate from below the first line. To get an   estimate of the second line we need again the inequality
$|I+B^\s(\g    )|\le Cr^{1-s}$.
Eventually we get
\[
  \begin{aligned}
| E^\s(h)-E^\s(h')| &  \ge \{C_0^{-1}|\lambda_\I(x)| \, r^{(n-1)(s-1)} - C_0 r^{1-s} \delta\}|h-h'|,
\end{aligned}
\]
for any $\s<\r$ and $|h-h'|<\r$.
   The proof is concluded as soon as we choose $\d=\d(\I,x,r)$ small enough and let $\sigma\to 0$.

This argument shows that the map is locally biLipschitz, as desired.
\end{proof}

\subsection{Pullback of vector fields through scaling maps}\label{cinquantadue}
 Given an $\eta$-maximal
  triple $(\I, x,r)$, for vector fields of the class ${\mathcal A}_s$ we can define, as in \cite{TW}, the ``scaling map''
 \begin{equation}\label{sferetta}
  \Phi_{\I,x,r}(t)= \exp \Big(\sum_{j=1}^n t_j r^{\ell{i_j}} Y_{i_j}\Big)x,
 \end{equation}
 for small
  $|t|$. The dilation $\d_r^I(t): = (t_1 r^{\ell_{i_1}}, \dots,t_n r^{\ell_{i_n}} )$ makes the natural  domain of
  $\Phi_{I,x,r}$ independent of $r$.  Observe the property $\|\d_r^\I t\|_\I=
r\|t\|_\I$.
 It turns out that, if  $\widehat X_{k} $  ($k=1,\dots, m$) denotes      the  pullback  of
 $r X_{k}$ under $\Phi_{\I,x,r}$, then $\widehat X_1,\dots, \widehat X_m$
  satisfy the
 H\"ormander condition in an uniform way.
This fact enables the authors in \cite{TW} to give several simplifications
 to the arguments in \cite{NSW}.

We can also consider the scaling map   associated with our exponentials.
  Namely,
\begin{equation}
 S_{\I,x,r}(t):= \exp_*(t_1 r^{\ell_{i_1}}Y_{i_1})\cdots \exp_*
 (t_n r^{\ell_{i_n}}Y_{i_n})= E_{\I,x}(\d_r^\I t),
\end{equation}
It will be proved in Subsection  \ref{principale} that, if $(I,x,r)$ is $\eta$-maximal, then $S$ is one-to-one on the
set $\{\|t\|_\I\le \e_0 \eta\}$.
If we assume that the original vector fields are of class $C^1$, see Remark \ref{ciuno}, thus, we may define,  for all
$i\in\{1,\dots,q\}$ the
vector fields
$ \wh Y_j:  = S_*^{-1}(r^{\ell_i} Y_i).
$

Theorem \ref{ema}  thus becomes
 \begin{proposition}\label{scalate}Let $X_1, \dots, X_m$ be vector fields in
$\A_s$.
 Let  $(\I,x,r)$ be an  $\eta-$maximal triple and let $S:  = S_{ \I,x,r}$ be the
associated scaling map.
  Then $S\big|_{ Q_I(\e_0 \eta)}$ is a locally biLipschitz map
  and for a.e.   $t\in Q_I(\e_0 \eta)$ we may write
\begin{equation}\begin{aligned}\label{ghmm}
 S_*(\p_{t_j})& = r^{\ell_{i_j}}
 Y_{i_j} (S(t)) + \sum_{k=1}^{n}
\wh  b_j^k r^{\ell_{i_k} }Y_{i_k} (S(t)) ,
\end{aligned}\end{equation}
where the functions $\wh b_j^k$ satisfy
\begin{equation}\label{fails}
|\wh  b_j^k |\le \frac{C}{\eta} \|t\|_\I \qquad\text{for a.e.  $t\in
Q_\I(\e_0\eta)$.}
\end{equation}
 Moreover, if $S$ is $C^1$ and we  write
$\wh  Y_{i_j} = \p_{t_j} + \sum_{k=1}^n {a_j^k}(t)\p_{t_k}$, then
\begin{equation}\label{brutt}
|a_j^k(t)|\le \frac{C}{\eta}\|t\|_\I  \qquad\text{for all $t\in Q_\I(\e_0\eta)$}.
\end{equation}
 \end{proposition}
\begin{proof} Formula \eqref{ghmm}   is just Theorem   \ref{ema}.
The proof of  \eqref{brutt} is a consequence of    \eqref{fails} and
of the following elementary fact:
given a
 square  matrix $B\in \R^{n\times n}$  with   norm $|B|\le \frac12 $,   we may
write $(I_n+B)^{-1} = I_n+A $, and
$
 |A|= \big|\sum_{k\ge 1}(-B)^k \big|\le 2|B|.
$
\end{proof}

In the framework of our almost exponential maps,
estimate  \eqref{brutt} is sharp, even for smooth vector fields. The better estimate
$Y_{i_j}(t) = \p_j + \sum_{k }  {a_j^k}(t)\p_{k}$ with $|a_j^k(t)|\le C |t|$, obtained in
 \cite{TW} for maps of the form \eqref{sferetta}, generically fails for $S$, as the following example shows.

\begin{example}\label{wright}
Let  $X_1=\p_1,$ $X_2= a(x_1)\p_2$ with $a(s)= s + s^2$, or any   smooth
function  with    $a(0)=0$ and
$a'(0)\neq 0\neq a''(0) $. A computation shows that
\[
\exp_*(h[X_1, X_2])(x_1, x_2) = \big(x_1, x_2+ \big\{a(x_1+|h|^{1/2})-
a(x_1)\big\} |h|^{-1/2}h \big) .
\]
Therefore, at $(x_1, x_2)=(0,0)$, for small $r$, we must choose the maximal  pair of
commutators $X_1, [X_1, X_2]$ and  we
have
\begin{equation*}
\begin{aligned}
 S (t_1, t_2)& =
  \exp_*(t_1 rX_1)\exp_*(t_2 r^2 [X_1, X_2])(0,0)= \big ( t_1 r,  a(r|t_2|^{1/2}) |t_2|^{-1/2} t_2r  \big).
\\&= \big ( t_1 r,   t_2 r^2 +   |t_2|^{1/2} t_2 r^3\big).
\end{aligned}
\end{equation*}
Therefore,
\begin{equation*}
\begin{aligned}
\wh X_1 = \p_{t_1}, \qquad \wh X_2 = \frac{t_1+r t_1^2}{1+\frac 32 r
|t_2|^{1/2}}\p_{t_2}, \qquad
\widehat{  [X_1, X_2] } =  \frac{1+2rt_1}{1+\frac 32 r
|t_2|^{1/2}}\p_{t_2}.
\end{aligned}
 \end{equation*}
Clearly the formula $\wh{[X_1, X_2]}= \p_{t_2} + O(|t|)$ cannot hold, but \eqref{brutt} holds.  Observe
also that, writing  $\wh{[X_1, X_2]} = \widehat f_{(1,2)}\cdot\nabla, $ we have
$ \displaystyle\sup_{t\in U} \big|\wh X_2 \wh f_{(1,2)}\big|\simeq \sup_{t\in
U}|t_2|^{-1/2}=
+\infty,
$ for any neighborhood $U$ of the origin.
Therefore, the vector fields $\widehat X_1, \widehat X_2$ do not even  belong to the class $\A_2$.
\end{example}

\subsection{Ball-box theorem}
\label{principale}  Here we give our main result.  We keep the notation from Subsection \ref{unotre}.

\begin{theorem}
\label{distro}Let $X_1, \dots, X_m$  be H\"ormander vector fields of step $s$ in the class  $\A_s$.  There are $r_0,\wt r_0, C_0>0$,  and for all $\eta\in(0,1)$ there are $\e_\eta, C_\eta>0$ such that:
\begin{itemize}

\item [(A)]
 if  $(\I,x,r) $ is $ \eta-$maximal for some $x\in K$, $r\le r_0$, then,  for any   $\e\le\e_\eta$, we have
\begin{equation}\label{spia}
  E_{I,x}(Q_\I(\e r))\supset B_\rho(x, C_\eta^{-1}   \e^s r);
 \end{equation}
 \item [(B)] if  $(\I,x,r) $ is $ \eta-$maximal for some $x\in K$, $r\le \wt r_0$, then
  the map $E_{\I,x} $ is one-to-one on the set $Q_\I(\e_\eta r)$.
\end{itemize}
\end{theorem}
\begin{remark}\label{elleuno}
  Observe that in the right-hand side of inclusion \eqref{spia} we use the distance $d_\r$. Therefore, a standard
 consequence
  of  \eqref{spia}
   is
    the well known  property
  $B(x,r)\supset B_{E}(x, C^{-1}r^s)$, for any  $x\in K$, $r<r_0$. See \cite{FP}.
\end{remark}
\begin{remark} \label{regolarita} In the paper \cite{TW}  the authors use
the exponential maps in \eqref{dinage}.
  If the vector fields have  step $s$, then
their method requires that the commutators of length $2s$ are at
least continuous. (Here, we specialize \cite{TW} to the case $\e=1$ and we
do not discuss the higher regularity estimate   \cite[Eq.~(2.1)]{TW}.)
This appears in the proof of (22) and (23) of
  \cite[Proposition 4.1]{TW}.
 Indeed in equation (29), the
commutator $[X_w, X_{w_j}]$ must be written as a linear combination
of commutators   $ X_{w' }$, where for algebraic reasons it must be
$\abs{w'} = \abs{w} + \abs{w_k} $. If
$\abs{w} = \abs{w_k} = s$, then commutators of degree
$2s$  appear.  A similar issue appears for $[Y_{w_i}, Y_{w_j}]$ at
the beginning of p.~619.
\end{remark}

\begin{remark}\label{errezero}
The reason why we introduce two different   constants $r_0$ and $\wt r_0$ is that $C_0,\e_0$ and $r_0$ depend only on $L$ and $\nu$ in \eqref{lipo} and \eqref{horma}  (together with universal constants, like $m,n$ and $s$). The constants  $\e_\eta$ and $C_\eta$ depend on $\nu, L$ and $\eta$ also.    We do not have a control of  $\wt r_0$  (which appears only in the injectivity statement)
 in term of $L$ and $\nu$. This is a delicate question because of the covering argument implicitly contained in
 \cite[p.~132]{NSW} and described in \cite[p.~230]{M}.
  Below we provide a constructive procedure to provide a lower bound  for $\wt r_0$ in term of the functions
  $\lambda_\I$.
See p.~\pageref{erresette}. This can be of some interest in view of applications of our results to nonlinear problems.
 \end{remark}
 \begin{remark}
   \label{convesso} The proof
  of the injectivity result would be considerably simplified if we could prove  (uniformly in $x\in K, r<r_0$)
    an equivalence between the balls and their convex hulls, i.e.
  $
   \mathrm{co} B(x,r)\subset B(x, Cr),
  $
  which is reasonable for diagonal vector fields (see \cite[Remark 5]{SW})   or
   a ``contractability'' property of the ball $B(x, r)$ inside   $B(x,
Cr)$. See
  \cite[Definition 1.7]{Se}.
    Unfortunately, in spite of their reasonable aspect, both these conditions seem quite difficult to prove
   in our situation.
It seems also that the clever argument in
\cite[p.~622]{TW}  can not be adapted to our almost exponential maps.
\end{remark}

In the proof of inclusion \eqref{spia},
we follow the argument in \cite{NSW,M}.
Before giving  the proof, we  need to show that some constants in the proof  actually  depend only on $L$ and $\nu$ in \eqref{lipo} and \eqref{horma}.
Basically, what we need is contained in Corollary \ref{765} and in the following
Lemma. See \cite[p.~129]{NSW}.
 \begin{lemma} \label{nove}
Assume
 that $(\I,x,r)$ is $\eta-$maximal for     vector fields $X_j$ in
$\A_s,$ $x\in K$, $r\le r_0$.
  Let   $\wt\sigma>0$  be  such that $(I, x, r)$ is $\eta$-maximal for the
mollified $X_j^\s$ for all $\s\le \wt \sigma$. Let
  $U\subset Q_\I( \e_\eta r ) $, where $\e_\eta $ comes from Theorem \ref{ema},
   and assume that a $C^1$ diffeomorphism   $\psi=(\psi^1,\dots,\psi^n): E^\s(U)\to U$ satisfies
$
 \psi(E^\s(h))=h,$ for any $h\in U.
$
Then we have the estimate
$|U_j^\s\psi^k(E^\s(h))|\le C r^{d_k-d_j},  $    for all $h\in U$,   where   $C$ is independent of $\sigma$.
\end{lemma}

\begin{proof}
It is convenient to work with the map $S^\s(t):= E^\s(\d_r t)$, so that $\phi := \delta_{1/r}\psi $ satisfies
$t=\phi(S^\s(t)) $, for all $t\in V:=\d_{1/r} U $. The chain rule gives
$ d\phi(S^\s(t)) dS^\s(t)= I$, for all $t\in V$.
 But, by \eqref{ghmm} we have  $dS^\s(t)= [r^{d_1}U_1^\s(S^\s),\dots,
r^{d_n}U_n^\s(S^\s)](I+\wh B^\s(t))$,
where  $|\wh B^\s(t)|\le\frac{ C}{\eta} \|t\|_I$, if $\|t\|_I\le \e_0\eta$.
  Therefore
 we may write
\[ d\phi(S^\s) [r^{d_1}U_1(S^\s),\dots, r^{d_n}U_n(S^\s)] = I+A^\s,
\] where, as in the proof of Proposition \ref{scalate}, $|A^\s(t)|
\le 2 |\widehat
B^\s(t)|$.
This implies that
$
|   r^{d_j} U^\s_j\phi^k(S^\s(t))|\le C $ and ultimately that $|   r^{d_j-d_k} U^\s_j\psi^k |\le C,
  $
as desired. \end{proof}

\begin{proof}[Proof of Theorem \ref{distro}, (A)]
Since the vector fields $Y_j$ are not Euclidean Lipschitz
continuous, if $\ell_j=s$, we do not know whether  or not  any point
in a
 $\rho$-ball of the $Y_j$ can be approximated by  points in the analogous ball
of the mollified $Y_j^\s$.
In order to avoid this
problem,  observe the  inclusion   $B_\r(x, r)\subset   B_{\wt\rho}(x, Cr)$ where $C$ is absolute and the distance $\wt\rho$
is defined using the family $\{Y_j : \ell_j\le s-1, \p_k :k=1,\dots, n\}$, where we assign to the vector fields $\p_k$
maximal weight $s$. Therefore, we will prove the inclusion using the distance
$\wt\rho$, which is defined by Lipschitz vector fields.

      Let   $(\I,x,r)$ be a $\eta-$maximal triple for the  original vector
fields $X_j$ and let $\wt \s$ be as in Lemma \ref{nove}.
Let $y\in B_{\wt\r}(x, C_\eta^{-1}   \e^s r)$, where $\e\le \e_\eta$, and
$\e_\eta$ comes from statement   (A),   while  $C_\eta$ will be discussed below.
Thus, $y=\gamma(1)$, where
$
 \dot\gamma  =\sum_{\ell_j\le s-1} b_j Y_j (\gamma ) + \sum_{i=1}^n \tilde b_j
\p_i(\gamma)$
 a.e. on $ [0,1],
$
 with
 $|b_j(t)|\le( C_\eta^{-1}   \e^sr)^{\ell(Y_j)}$ and $|\tilde b_i(t)|\le(
C_\eta^{-1}  \e^sr)^{s}$   for a.e.  $t\in[0,1]$.
Let also $y^\s\in B_{\wt\rho}(x, C_\eta^{-1}   \e^s
r)$ be an approximating family,
$y^\s=\gamma^\s(1)$, where
$
 \dot\gamma^\s =\sum_{\ell_j\le s-1} b_j Y_j^\s(\gamma^\s) + \sum_{i=1}^n \tilde b_j  \p_i(\gamma^\s)$
 a.e. on $ [0,1].$  Observe that $y^\s\to y$, as $\s\to 0$.

\medskip\noindent {\it Claim. } If $C_\eta $ is large enough, then for any $\s\le\wt\s$
  there is a lifting map $\theta^\s(t),$
 $t\in[0,1]$, with $\theta^\s(0)=0$ and such that
\begin{equation}
 \label{ottosette} E^\s(\theta^\s(t))= \gamma^\s(t) \quad \text{and}\quad
 \|\theta^\s(t)\|_\I< \e   r
 \quad \text{for all }   t\in[0,1].
\end{equation}
Once the claim is proved, the surjectivity statement follows.

 To prove the claim the key estimate we need is the following.
Let $U\subset Q_\I(\e_\eta r)$, $\s\le \wt\s $ and assume that a $C^1-$ diffeomorphism $\psi =(\psi^{1 },
\dots,\psi^{n })$ satisfies locally $\psi (E^\s(h))=h$, for all $h\in U  $,
where, for some $t\in[0,1]$,  $E^\s(U)  $ is a  neighborhood  of $\gamma^\s (t)$.
 Then, for $\mu=1,\dots,n$ and for all $\t$ close to $t$
\begin{equation}\label{giovanni}
 \begin{aligned}
& \Big|\frac{d}{d\t} \psi^\mu (\gamma^\s(\t))\Big|
\\&
=  \Big|\sum_{\ell_j\le s-1 }  b_j(\t)Y^\s_j\psi^\m(\gamma^\s(\t))
+\sum_{i=1}^n\tilde b_i (\t)\p_j\psi^\mu(\gamma^\s(\t))\Big|
\\&=  \Big|\sum_{\ell_j\le s }  b_j(\t)\sum_{k=1}^n a_j^k(\gamma^\s(\t)) U^\s_k \psi^\m(\gamma^\s(\t))
\\&\qquad +\sum_{i=1}^n
\tilde b_i (\t)   \sum_{k=1}^n \tilde a^k_i(\gamma^\s(\t)) U^\s_k \psi^\m(\gamma^\s(\t))
\Big|
\\&\le \sum_{j,k}
C(C_\eta^{-1}  \e^s r)^{\ell(Y_j)} \cdot \frac{C}{\eta} r^{d_k-\ell(Y_j)} \cdot Cr^{d_\mu-d_k}
\\&\qquad + \sum_{i,k}
C(C_\eta^{-1}  \e^s r)^{s} \cdot \frac{C}{\eta} r^{d_k-s} \cdot Cr^{d_\mu-d_k}
\\&\le\frac{C C_\eta^{-1}}{\eta} \e^s  r^{d_\mu}\le \frac{C C_\eta^{-1}}{\eta}(\e r)^{d_\mu}. \end{aligned}
\end{equation}
The constant $C_\eta $ will be chosen below, while $C$   depends on $L,\nu$,
 in force of Corollary \ref{765} and Lemma \ref{nove}.
We used the estimate $\p_i =\tilde a_i^k U_k^\s$ with $\abs{\tilde a_i^k}\le \frac
C\eta r^{d_k-s}$, which follows from Lemma \ref{rough} and Corollary \ref{765}.

With estimate \eqref{giovanni} in hands we can prove the claim along the lines of \cite[p.~228]{M}. Here is a sketch of the argument.

{\it Step 1.  } If $C_\eta$ is large enough,   then,  if $\theta^\s(t)$ satisfies $E(\theta^\s (t))=\gamma^\s(t)$
 on $[0,\bar t]$, for some $\bar t\le 1$, then $\|\theta^\s(t)\|_\I<\frac12\e r$, for any $t\le \bar t$.

To prove Step 1, assume by contradiction that the statement is false.
There is $\wt t\le \bar t$ such that $\|\theta^\s(t)\|_\I<\frac12\e r$ for all $t<\wt t$ and
  $\|\theta^\s(\wt t)\|_\I=\frac12\e r$. Then for some $\mu\in\{1,\dots, n\}$,  we have
\begin{equation}
 \begin{aligned}
\Big(  \frac 12 \e r\Big)^{d_\m}& = |\theta^\s_\m(\wt t)|
\le \frac{C C_\eta^{-1}}{\eta} (\e r)^{d_\mu}.
 \end{aligned}
\label{h5}
\end{equation}
This estimate can be obtained writing locally
$\theta^\s(t)=\psi(\gamma^\s(t))$ and using \eqref{giovanni}.
If we choose $C_\eta $   large enough to ensure that $\frac{C C_\eta^{-1}}{\eta} <(\frac12)^s$,  then
   \eqref{h5}
 can not hold and we have a contradiction. This ends the proof of   {\it Step 1. }

{\it Step 2. } There exists a path  $\theta^\s$  on $[0,1]$ satisfying \eqref{ottosette}.

The proof of Step 2   can be done as in \cite[p.~229]{M} by a very classical
argument, involving an upper bound ``of Hadamard type''
$\|dE^\s(\theta^\s(t))^{-1}\|\le C$, which holds uniformly in  $t $.

The proof of the statement (A) is concluded.
\end{proof}

Before proving part $(B)$ of Theorem \ref{distro}, we need the following rough injectivity statement.
\begin{lemma}
\label{cofg} Let $x\in K$ and  $\I$ such that $\lambda_\I(x)\neq 0$. Then the function $E_{\I,x}$
is one-to-one on the set
  $   Q_\I(C^{-1}|\lambda_\I(x)|)$.
  \end{lemma}

\begin{proof}
Observe first that for all $j=1, \dots, n$ and small $\s$, we have
\begin{equation}\label{sudf}\begin{aligned}
\Big| \frac{\p}{\p h_j} & E^\s(h)- U^\s_j(x) \Big|
\\&\le \Big|\frac{\p}{\p h_j} E^\s(h)- U^\s_j(E^\s(h))\Big|+|U^\s_j(E(h))- U^\s_j(x)|
\le C\|h\|_\I,
\end{aligned}\end{equation}
 by estimates \eqref{wiz3},  \eqref{pso3}  and the $d$-Lipschitz continuity of $U^\s_j.$

 Fix $h,h'\in Q_\I(C^{-1}|\lambda_\I(x)|)$ and let $\g(t)=h'+t(h-h')$. Then
\[
 \begin{aligned}
| E^\s(h)-E^\s(h')|&= \Big|  \int_0^1 dE^\s(\g)(h-h')dt\Big|
   \\&\geq  |dE^\s(0)(h-h')|- \Big| \int_0^1\big\{ dE^\s(\g)-dE^\s(0)\big\}dt(h-h')\Big|
 \\&  \ge \big\{ C^{-1}|\lambda_\I^\s(x)| - C\max\{\|h\|_\I,\|h'\|_\I\}\big\}|h-h'|.
 \end{aligned}
\]
by \eqref{sudf} and because $dE^\s(0)=U^\s(x)= [U^\s_1(x) ,\dots, U^\s_n(x)] $ has determinant $\lambda_\I^\s(x)$.
The proof is concluded by letting $\s\to 0$.
\end{proof}

 As announced in Remark \ref{errezero},
 we provide a constructive procedure for the ``injectivity radius'' $\wt r_0$ in
Theorem \ref{distro} in
  term of the functions $\lambda_\I$.
 Compare \cite[p.~229-230]{M}.

Denote by $D_1,\dots, D_p $  all the  values attained by   $\ell(\I)$, as $\I\in\ii$. Assume that $D_1<\cdots <D_p$ and  introduce the notation:
  \begin{equation}\label{gtk}
   \sum_{\I}|\lambda_\I(x)| r^{\ell(\I)} = \sum_{j=1}^p r^{D_j} \sum_{\ell(\I)=D_j} |\lambda_\I(x)|  =
 :  \sum_{j=1}^p r^{D_j} \mu_j(x),
  \end{equation}
 where $\mu_j$ is defined by \eqref{gtk}.
  Let $\Sigma_1:=K$ and, for all  $k=2,\dots,p$,
  \begin{equation*}
  \begin{aligned}
   \Sigma_k: =\{x\in K : \mu_j(x)=0 \text{ for any } j=1,\dots, k-1\}.
  \end{aligned}  \end{equation*}
   Observe that $\Sigma_1= K \supseteq
  \Sigma_2\supseteq\cdots\supseteq\Sigma_p.$
  Let  $  x\in K$.  Take
  $
   j( x)=\min \{j\in\{1,\dots,p \}: \m_j(x)\neq 0\}.
  $
 Then choose   $\I_{  x} \in\ii$ such that
 $
 |\lambda_{\I_{  x}}(  x)|$ $=\displaystyle{\max_{\ell(\J)=  D_{j(x)}}|\lambda_\J(x)|}
 $ is maximal. Therefore, we have
 $
   |\lambda_{\I_{ x}}( x)|\simeq \m_{j(x)}(x),
 $
 through  absolute  constants.

From the construction above we get the following proposition.
  \begin{proposition}\label{stst}
There is $C>1$ such that, letting  $    r_{  x}:= C^{-1} |\lambda_{\I_{
x}}( x)| $  for all   $x\in K$, then:
   \begin{itemize}
   \item [(1)]
  we have
\begin{equation}\label{erty}
 | \lambda_{I_{ x}}( y)| r_x^{\ell(I_{ x})}> C^{-1}  \Lambda(y, r_x) \qquad \text{for all } y \in B ( x,\e_0 r_{ x}) ;
\end{equation}
\item[(2)] the map $h\mapsto E_{\I_{ x}}( y  , h)$ is one-to-one on the set $Q_{\I_{ x}}( r_{x})$, for any $y\in
    B ( x,  \e_0 r_{ x})$.
   \end{itemize}
  \end{proposition}
Observe that Proposition \ref{stst} is far from what we need,
because it may be $\inf\limits_K r_x=0$,
  (for example this happens in  the elementary situation $X_1=\p_1, X_2= x_1\p_2$.)

\begin{proof}   We first prove (1) for
$y=x $. Namely we show that
\begin{equation}\label{ertyu}
 |\lambda_{I_{ x}}(  x)| r^{\ell(I_{ x})}\geq |\lambda_J(   x)| r^{\ell(J)}
\qquad
  \text{for all $ J\in\ii$ \quad  $ r\in[0,r_x],$}
\end{equation}
 where $r_x=
  C^{-1}|\lambda_{I_x}(x)|$, as required.   Let $J\in\ii$.
 If  $\lambda_J(x)=0$, then
  \eqref{ertyu} holds for all $  r>0$.
  If instead
  $\lambda_J( x)\neq 0$,
   by the choice of $I_x$ it must be   $\ell(J)=\ell(I_{ x})$
    or  $\ell(J)>\ell(I_{  x})$. If   $\ell(J)=\ell(I_{ x})$, then
   \eqref{ertyu} holds for any   $r>0$,  because  $|\lambda_{\I_{ x}}(  x)|$
is maximal, by the  construction above. If   $\ell(\J)
>\ell(\I_{x})$,
then
 \[\begin{aligned}
 & |\lambda_{\J}( x)| r^{\ell(\J)} \le|\lambda_{\I_{ x}}
 ( x)|r^{\ell(\I_{ x})} \quad\Leftarrow \quad C r^{\ell(\J)- \ell(\I_{
 x})}\le |\lambda_{\I_{ x}}( x)|
\quad  \Leftarrow\quad  r\le C^{-1}|\lambda_{\I_{ x}}(
x)|.\end{aligned}\] Thus \eqref{ertyu} holds for any  $r\le r_{  x} $,
where $r_x$ has  the required form.

The proof of  (1) for   $y\neq x $ follows from   Theorem
\ref{dapro}.

Finally, to prove  (2)  observe that, in view of
Lemma \ref{cofg}, the map $h\mapsto E_{I_x}(y,h)$ is one-to-one on
$Q_{I_x}(C^{-1}|\lambda_{I_x}(y)|)$. But Theorem \ref{dapro}, in
particular \eqref{aceto} show that, if $d(x,y)\le \e_0 r_x$, then
$|\lambda_{\I_x}(y)|$ and $|\lambda_{\I_x}(x)|$ are comparable. This
concludes the proof.
   \end{proof}

\noindent{\it Proof of Theorem \ref{distro}, (C).}
    Let $p_1\le p$  be the largest  integer such that $\Sigma_{p_1}\ne\varnothing$.
      Then define the ``injectivity radius''
\begin{equation}
r_{(p_1)} :=\min_{x\in\Sigma_{p_1}} r_x =\min_{x\in
\Sigma_{p_1}}C^{-1} |\lambda_{\I_x}(x)|\ge C^{-1}
 \min_{x\in \Sigma_{p_1}}\mu_{p_1}(x)>0.
\end{equation}
Denote also
\[\Omega_{p_1}=\bigcup_{x\in \Sigma_{p_1}}
\Omega' \cap B(x, r_{(p_1)}),\] where the open set $\Omega'$ was
introduced before \eqref{lipo}. Recall that all metric balls $B(x, r)$ are open, by the already 
accomplished
Theorem \ref{distro}, part (A). Then, by Proposition \ref{stst},
 for
  any $y\in \Omega_{p_1}$
  there is $  x\in\Sigma_{p_1}$ such that the map $h\mapsto E_{\I_{  x}}(y, h)$ is one-to-one on
 $Q_{I_{ {x}}}(\e_0 r_x)$
  and $(\I_x, y, r_x)$ is $C^{-1}$-maximal.
   Recall that $r_x\ge r_{(p_1)}$, on $\Sigma_{p_1}$.

\label{erresette} Next let $p_2< p_1$ be the largest number such that $K_{p_2}:=
\Sigma_{p_2}\setminus\Omega_{p_1}\neq\varnothing $.  Then, let
\[
r_{(p_2)}:=\min_{x\in \Sigma_{p_2}\setminus\Omega_{p_1}} r_x\ge C^{-1}
\min_{x\in \Sigma_{p_2}\setminus\Omega_{p_1}} \mu_{p_2 }(x)>0.
\]
We may claim that   for any $y\in  \Omega_{p_2}:=\bigcup_{x\in K_{p_2}}\Omega'\cap B(x, r_{(p_2)} ),
$
there is $  x\in K_{p_2}$ such that the map $h\mapsto E_{\I_{
x}}(y,h)$ is one-to-one on the set $Q_{I_{ x}} (\e_0 r_x)$ and
$(\I_x,y,r_x)$ is $C^{-1}$-maximal.

Iterating the argument, and letting $\wt r_0 = \min \{ r_{(p_k)} \}$  we conclude that for any $x\in K$ there is  a
$n$-tuple $I_0=I_0(x)$, and $\r_0=\r_0(x)\geq\wt r_0$ such that  $E_{\I_0}(x,\cdot)$ is one-to-one on the set
$Q_{I_0}(\e_0 \r_0)$ and $(I_0,x,\r_0)$ is $C^{-1}$-maximal.  Clearly,  $\I_0$ can be different from $\I_x$.
This is
the starting point
 for the proof of the injectivity statement, Theorem \ref{distro}, item (B).

From now on  $\I,$ $x\in K$ and $r<\wt r_0$ are fixed and $(\I,x,r)$
is $\eta-$maximal, as in the hypothesis of (B).  Let $I_0$ and $\r_0$ be the $n$-tuple
  and the injectivity radius
associated with
$x$ by the argument above. Recall that  $\r_0\ge \wt r_0$. Arguing as in
\cite[p.~230]{M},
 see also
\cite[p.~133]{NSW}, we may find a sequence of $n-$tuples $I=I_N,
I_{N-1}, \dots, I_1, I_0 $ and correspoding numbers $0\le
\r_{N+1}<\r_N<\cdots< \r_0$, with $\r_0\ge \wt r_0$,
$r\in[\r_{N+1}, \r_N]$ such that   for any $j=0,1,\dots,N-1$,
\begin{equation}
\label{sote} |\lambda_{\I_j}(x)| \r^{\ell(\I_j)}\ge   \eta \Lambda(x,
\r),\quad\forall \;\r\in[\r_{j+1}, \r_j].
\end{equation}

In order to show that $E_\I = E_{\I_N}$
is one-to-one on the set
$Q_{\I}(\e_\eta r) $, for some   $\e_\eta>0$,
we start by showing
that $E_{\I_1}$ is one-to-one on the set $Q_{\I_1}( \e'_\eta \r_1 ),$ for a suitable $\e'_\eta$.
What we know is that $E_{\I_0}$ is one-to-one on the set
$Q_{\I_0}(\r_0)$. We also know that \eqref{sote} holds for
$j=0,1$ and $\r=\r_1$. Therefore, applying twice  \eqref{spia}, we have
\begin{equation}\label{etos}
\begin{aligned}
E_{\I_1}(Q_{\I_1}(\e_\eta \r_1))& \supseteq  E_{\I_0}(Q_{\I_0}(C_\eta^{-1}  \r_1) )
\supseteq  E_{\I_1}\big(Q_{\I_1}(
\e'_\eta  \r_1)\big).\end{aligned}
\end{equation}
Assume by contradiction  that $E_{\I_1}(h) = E_{\I_1}(h')=y$ for some $h, h'\in
Q_{\I_1}(\e_{\eta}'\r_1)$. Let $r(t)= h'+t(h-h')$, $t\in[0,1]$ be
the line segment connecting $h$ and $h'$. Let also $\gamma(t)=
E_{\I_1}(r(t))$. Since $E_{\I_0}$ is one-to-one (actually a $C^1$ diffeomorphism on
its image), we may contract $\g$ to a point just by  letting
$
q(\lambda,t)= E_{\I_0}\big(\lambda E_{\I_0}^{-1}(y) + (1-\lambda) E_{\I_0}^{-1}(\g(t))\big),$ where $  (\lambda,t)\in [0,1]\times[0,1].
$   Observe that $q$ is continuous
   on $[0,1]^2$, and $q(\lambda,t)\in Q_{\I_1}(\e_\eta \r_1)$, by \eqref{etos}.
    Moreover $q(0,t)=\g(t)=E_{\I_1}(r(t))$ and $q(1,t)=y$, for any
    $t\in [0,1]$.
   By standard properties of  local diffeomorphisms
    we may claim that there is a continuous lifting $p:[0,1]^2\to Q_{\I_1}(  \e_\eta  \r_1)$
   such that $E_{\I_1}(p(\lambda,t))=q(\lambda,t)$ and $p(0,t)= r(t)$ for all $\lambda$ and  $t\in[0,1]$.
Next observe that both the maps $\lambda\mapsto E_{\I_1}(p(\lambda,1))$ and $\lambda \mapsto E_{\I_1}(p(\lambda,0))$ are constants on $[0,1]$. Therefore, since $E_{\I_1}$ is a  local diffeomorphism, both $\lambda\mapsto p(\lambda,0)$ and  $\lambda\mapsto p(\lambda,1)$
must be constant. In particular $p(1,1)=p(0,1)=h'$ and $p(1,0)=p(0,0)=h$. Finally observe that the path $t\mapsto p(1,t)$ must be constant, because $E_{\I_1}(p(1,t))=y$ for all $t\in[0,1]$. Therefore we conclude that $h=h'$.

    Then we have
     proved that $E_{\I_1}$ is one-to-one
     on $Q_{\I_1}(\e_\eta' \r_1)$. Iterating the argument at most $N$ times,
     we get  the proof of  statement (B) of Theorem \ref{distro}.

\section{Examples}\label{esempi}
\begin{example}[Levi vector fields] \rm
In order to illustrate the previous procedure to find $\tilde r_0$
we exhibit the following three-step example. In $\R^3$ consider the   vector
fields $X_1= \partial_{x_1}+a_1 \partial_{x_3}$ and $X_2=
\partial_{x_2}+a_2 \partial_{x_3}$. Assume that the vector fields belong to the
class $\A_3$. Let us  define $f=X_1a_2-X_2a_1$. Morover
assume that  $|f|+|X_1f|+|X_2f|\neq 0$ at every
point of the closure of a bounded set  $\Omega\supset  K=\overline{\Omega'}$. Assume also that  $f$ has some zero inside $K$.  This condition naturally arises in the
regularity theory for graphs of the form $\{ (z_1, z_2)\in \C^2 : \Im(z_2) =\phi(z_1, \bar z_1, \mathrm{Re}(z_2))\}$  having
some first order zeros. See \cite{CM}, where the smoothness of $C^{2,\alpha}$ graphs with prescribed smooth Levi curvature is proved.

 In this situation we have  $n=3, m=2, s=3$ and
$Y_1=X_1, Y_2=X_2, Y_3=[X_1, X_2]=f\p_{x_3},
Y_4=[X_1,[X_1,X_2]]=(X_1f-f\p_{x_3}a_1) \p_{x_3},
Y_5=[X_2,[X_1,X_2]]=(X_2f-f\p_{x_3}a_2) \p_{x_3}.$ Thus, $q=5$ and
\[
\begin{aligned}
\lambda_{(1,2,3)}=&f,   &d( 1,2,3 )=4,\\
\lambda_{(1,2,4)}=&X_1f-f\p_{x_3}a_1,   &d( 1,2,4) =5,\\
\lambda_{(1,2,5)}=&X_2f-f\p_{x_3}a_2,   &d (1,2,5) =5.
\end{aligned}
\]
Let us put  $D_1=4, D_2=5$ and, by \eqref{gtk}, $\mu_1=|f|, \mu_2=|X_1f-f\p_{x_3}a_1|+
|X_2f-f\p_{x_3}a_2|.$ In this situation $\Sigma_1=K,$ $\Sigma_2=\{x\in K : \mu_1(x)=0
\}=\{x\in K : f(x)=0 \}.$ Hence, $r_{(2)}=\min_{x \in
\Sigma_2}r_x=\min_{x \in \Sigma_2} \max\{|X_1f(x)|, |X_2f(x)|\}>0.$  Let
$\Omega_2=\cup_{x\in \Sigma_2}\Omega'\cap B(x,r_{(2)})$, with
$\bar \Omega'=K$, and let $K_1=\Sigma_1\setminus \Omega_2.$ Since
$K_1 \subseteq \{x\in K: f(x) \neq 0\},$ if $K_1\neq \varnothing$ then
$r_{(1)} =\min_{x\in K_1}r_x=\min_{x\in K_1}|f(x)|>0.$ Finally, if
$K_1\neq \varnothing$ then $\tilde r_0=\min \{r_{(1)},r_{(2)}\}$,
while if $K_1=\varnothing$ then $\tilde r_0=r_{(2)}$.

\end{example}

\medskip

In next example we show a   subelliptic-type estimate for nonsmooth
vector fields. The argument of the   proof below is due to Ermanno Lanconelli (unpublished).
\begin{proposition}[H\"ormander--type estimate \cite{H}]
\label{ormai}  Let $ X_1,\dots, X_m$ be a family of vector fields of step $s$ and in  the class $\A_s$. Then, given $\Omega'\subset\subset\Omega$,  and $\e\in\left]0,1/s\right[$, there is $\wt r_0$ and $C>0$ such that
such that, for any $f\in C^1(\Omega)$,  \begin{equation}\label{orol}
[ f]_{\e}^2:=  \int\limits_{\substack{{\Omega'\times\Omega'},\\
   {d(x,y)\le \wt r_0}}}\frac{|f(x)-f(y)|^2}{|x-y|^{n+2\e }}dxdy\le C\int_\Omega \sum_j|X_jf(y)|^2 dy.
 \end{equation}
\end{proposition}
\begin{proof}
We just  sketch the proof, leaving some details to the reader.  For any $I\in\ii$,  let $\Omega_\I
:= \{ x\in\Omega :  \I_0(x) =\I\} $, where $\I_0(x)$ comes from the proof of
Theorem \ref{distro}, together with $\r_0=\r_0(x)\ge \wt r_0$, see the
discussion before equation \eqref{sote}.   If $x\in \Omega_\I$, we have $B(x,
\r_0)\subset E_\I(x,Q_\I(C \r_0))$, where the biLipschitz map $E_\I$ satisfies
  $C^{-1}\le |\det dE_\I(x,h)|\le C$, for a.e.  $h\in Q_\I(C \r_0)$. Thus,
 \[
 \begin{aligned}\,
[f]_{\e}^2  = \int\limits_{\substack{\Omega'\times\Omega'\\ d(x,y)\le  \wt r_0}} \frac{|f(x)-f(y)|^2}{|x-y|^{n+2\e}}dxdy
  &\le \sum_{\I\in \ii}
   \int_{\Omega_\I}dx\int_{d(x,y)\le \r_0}dy\frac{|f(x)-f(y)|^2}{|x-y|^{n+2\e}}
\\ \le C \sum_\I &\int_{\Omega_\I}dx\int_{Q_\I(C \r_0)} dh \frac{|f(x)-f(E_\I(x,h))|^2}{|x-E_\I(x,h)|^{n+2\e }}.\end{aligned}
\]
Now observe that, arguing as in the proof of Lemma \ref{cofg}, we have   $|E_\I(x,h)-x|\ge C^{-1}|h|$, if $\|h\|_\I\le C  \r_0$.
Let $\d_0=\max_{x\in K}\r_0(x)$. Next we follow the argument   in \cite{LM}.  Write $E_\I(x,h)=\gamma_\I(x,h,T(h))$, where $\gamma_\I(x,h,t)$, $t\in[0, T(h)]$ is a \emph{control function}, with the properties described in \cite{LM}. \rm
Therefore
\[
 \begin{aligned}
  \ [f]_\e^2 &= \sum_\I\int_{\Omega_\I}dx\int_{Q_\I(C \delta_0)} dh \frac{|f(x)-f(E_\I(x,h))|^2}{|h|^{n+2\e}}
 \\&\le C\int_{Q_\I(C \d_0)} \frac{dh}{|h|^{n+2\e}} \int_{\Omega_\I}dx \Big|\int_0^{T(h)}
 dt|Xf(\g_\I(x,h,t))|\Big|^2
\\& \le C \int_{Q_\I(C \d_0)} \frac{dh}{|h|^{n+2\e}}\Big\{  \int_0^{T(h)}dt\Big(\int_{\Omega_\I}dx  |Xf(\g_\I(x,h,t))|^2 \Big)^{1/2}\Big\}^2
\\&\le C  \int_{Q_\I(C \d_0)} \frac{dh}{|h|^{n+2\e}} T(h)^2 \|Xf\|_{L^2(\Omega)}^2\le C\|Xf|_{L^2(\Omega)}^2,\end{aligned}
\]
because $x\mapsto \gamma_\I(x,h,t)$ is a change of variable,  by  estimate $T(h)\le \|h\|_\I\le |h|^{1/s}$ and the strict
inequality $\e<1/s.$\end{proof}
 The borderline inequality $\|f\|_{1/s}\le C\|Xf\|_{L^2}$, which
can not be obtained  with the argument above, was proved in the
smooth case by Rothschild and Stein \cite{RoS}.

\section{Proof of   Proposition \ref{schiaccia2}} \label{appendiamo}
\setcounter{equation}{0}

Here we prove  Proposition \ref{schiaccia2}.
 \label{schiaccia}
   By definition, \eqref{dhc} means that
 for all $j,k\in\{1,\dots, m\}$ and  $\abs{w}\le s-1$
 there is a bounded function $X_j (X_k f_w)$ such that for any test function
   $\psi\in C_c^{\infty}(\R^n)$,
  \begin{equation}\label{chillo}
  \int(X_k f_w)(X_j\psi)  = -\int\{X_j (X_k f_w)+\mathrm{div}(X_j)X_k f_w\}\psi.
  \end{equation}
If $D= \p_{j_1}\cdots\p_{j_p} $ for some $j_1, \dots, j_p\in\{1,\dots,n\}$ is
an Euclidean derivative, denote by  $|D|=p$ its order.
 It is understood that a derivative of order $0$ is the identity.

The first item of Proposition \ref{schiaccia2}   is a consequence of the following lemma:
\begin{lemma}\label{mp4}
     Let $X_1, \dots, X_m$ be vector fields in $ \A_s$.
     Then for any word  $w$ with $|w|\le s$ and for any Euclidean derivative $D$ of order  $ |D|=p\in\{0,\dots,   s
     -|w|\}$, we have
   \begin{equation}\label{mp3}
   \sup_K \left|D f_w^\s -(D f_w)^{(\s)}\right|\le C\sigma.
   \end{equation}
\end{lemma}
Note that, the case $p=0$ of \eqref{mp3} provides the proof of item 1 of Proposition \ref{schiaccia2}.

Observe also that, if $\abs{w}=s$, then  we have  $\abs{f_w-
f_w^\s}\le \abs{f_w- (f_w)^\s}+ \abs{f_w^\s- (f_w)^\s}$. Lemma
\ref{mp4} gives the estimates of the second term. The first one is
estimated by means of the continuity modulus of $f_w$, which is not
included in $L$ in \eqref{lipo}.

\begin{proof}[Proof of  Lemma \ref{mp4}]
We argue by induction on $|w|$. If $|w|=1$, then the left hand side of  \eqref{mp3} vanishes.
Assume that for some $\ell\in\{1, \dots, s-1\}$,    \eqref{mp3} holds for any
word
  $w $ of length $\ell$ and for each $D$ with $|D|\le s-\ell$.  Let
$v=kw $ be a word of length $|kw|= \ell+1$.
We must show that for any Euclidean derivative $D$ of order
$0\le |D|\le s-|v| $,  \eqref{mp3} holds.
We have  $f_v= X_k f_w-X_w f_k$ and $f_v^\s= X_k^\s f_w^\s-X_w^\s f_k^\s$. We first prove \eqref{mp3} when the order of $D$
satisfies $  1\le |D| \le s-|v|=s-\ell-1$, which can occur  only if $\ell\le s-2$ (in particular
this implies $s\ge 3$).
 The easier case is when $D$ is the identity operator and it will be proven
below.
\[\begin{aligned}
D f_v^\s -(Df_v)^{(\s)} &  = D\big( X_k^\s f_w^\s-X_w^\s f_k^\s\big)
                    -\big(D( X_k f_w-X_w f_k)\big)^{(\s)}
 \\&
 =  D\big( X_k^\s f_w^\s \big) -  (DX_k f_w)^{(\s)}
  -  \{D( X_w^\s f_k^\s )
                    -\big(DX_w f_k\big)^{(\s)} \}
\\&= : (A)-(B).  \end{aligned}
\]
Omitting summation sign on $\a=1,\dots,n$, we may write
\[
 \begin{aligned}
(A)& = (D f_k^\a)^{(\s)}\big\{\p_\a f_w^\s -(\p_\a f_w)^{(\s)}\big\} +
\big\{ (D f_k^\a)^{(\s)} (\p_\a f_w)^{(\s)}- \big( (D f_k^\a) \p_\a f_w\big)^{(\s)}\big\}
\\&+  (  f_k^\a)^{(\s)} \big\{        D\p_\a f_w^\s - (D\p_\a f_w)^{(\s)}\big\}
  +\big\{ (  f_k^\a)^{(\s)}  (D\p_\a f_w)^{(\s)} -  (  f_k^\a   D\p_\a f_w\big)^{(\s)} \big\}
\\& = :(A_{1})+(A_{2})+(A_{3})+(A_4).
 \end{aligned}
\]
The estimate $|(A_{1})|+|(A_3)|\le C\sigma$ follows from the induction assumption.
To estimate $(A_{4})$ observe that
\[
\begin{aligned}
|(A_{4})|&= \Big|\int\big\{ (f_k^\a)^{(\s)}(x) -f_k^\a(x-\s y)\big\}  D\p_\a f_w(x+\s y)\phi(y) dy\Big|\le C\sigma,
  \end{aligned}
\]
because $f_k$ is Lipschitz, while $ D\p_\a f_w\in L^\infty_{\loc}$. Indeed, since $|w|=\ell$,   $f_w\in W^{s-\ell,\infty}$. Moreover,  $D$ has length at most  $s-\ell-1$ so that $D\p_\a$ has lenght at most $s-\ell$.
The  estimate of $(A_{2})$ is analogous to that of $A_{4}$. Just recall that $D f_k^\a$ is Lipschitz and $\p_\a f_w$ is bounded.

Next we estimate $(B)$.
\[
 \begin{aligned}
  (B)&= D((f_w^\a)^\s \p_\a f_k^\s) - (D(f_w^\a \p_\a f_k))^{(\s)}
  \\&=  \{ D(f_w^\a)^\s - (Df_w^\a)^{(\s)} \} \p_\a f_k^\s +
  \big\{(Df_w^\a)^{(\s)} \p_\a f_k^\s - (( Df_w^\a)\p_\a f_k)^{(\s)}\big\}
\\&\quad + (f_w^\a)^\s D\p_\a f_k^\s -  ( f_w^\a(D\p_\a f_k))^{(\s)}
\\& = : (B_{1})+(B_{2})+(B_{3}). \end{aligned}\]
The term $ (B_{1})$ can be estimated by the inductive assumption.
Moreover,
\[
| (B_{2})|=\Big|\int \big\{ \p_\a f _k^\s (x)-\p_\a f_k(x- \s y)\big\}Df_w^\a(x-\s y)\phi(y) dy\Big|
\le C\sigma,
\]
because $\p_\a f _k$ is Lipschitz and $Df_w^\a\in L^\infty_{\loc} $.
Finally
\[
| (B_{3})|=\Big|\int \big\{ (f_w^\a)^{(\s)} (x)-f_w^\a(x- \s y)\big\}D\p_\a f_k(x-\s y)\phi(y) dy\Big|
\le C\sigma.
\]
Indeed, since $|w|\le s-2$, $f_w^\a $ is locally Lipschitz. Moreover,   since the length of the derivative $D\p_\a$ is at most $s-1$ and $f_k\in W^{s-1,\infty}_{\loc}$, we have  $D\p_\a f_k\in L^\infty$.

Next we look at the case where $D$ has length zero, i.e. $D$ is the identity operator. We have to estimate, for $v$ with $|v|\le s$, the difference $f_v^\s - (f_v)^{(\s)}$. Write $v= kw$, where $k\in \{1,\dots, m\}$.
Thus
\[
 \begin{aligned}
  f_v^\s - (f_v)^{(\s)}&= X_k^\s f_w^\s - X_w^\s f_k^\s - (X_k f_w)^{(\s)} + (X_w f_k)^{(\s)}
 \\&  = (f_k^\a)^{(\s)} \{ \p_\a f_w^\s - (\p_\a f_w)^{(\s)}\}
+\big\{ (f_k^\a)^{(\s)}    (\p_\a f_w)^{(\s)} -      ( f_k^\a\p_\a f_w)^{(\s)} \big\}
\\& \qquad - \{ (f_w^\a)^\s \p_\a f_k^\s - (f_w^\a  \p_\a f_k)^{(\s)} \}
\\&=(S_1)+ (S_2)+ (S_3).
\end{aligned}
\]
Now $(S_1)$ can be estimated by the inductive assumption. Moreover,
\[
| (S_2)| =\Big|\int  \{ (f_k^\a)^{(\s)} (x)    -       f_k^\a (x-\s y)\} \p_\a f_w(x-\s y) \phi(y) dy \Big|\le C\sigma,
\]
because $f_k^\a$ is locally  Lipschitz continuous  and $\p_\a f_w $ is locally bounded, since  $|w|\le s-1$.
Finally
\[
| (S_3)| =\Big|\int
\big\{(f_w^\a)^\s(x)- f_w^\a(x-\s y)\big\}\p_\a f_k(x-\s y)\phi(y)dy \Big|\le C\sigma,
\]
because  $f_w$ is Lipschitz and  $\p_\a f_k$  is bounded. This concludes the proof of the the first item of
Proposition \ref{schiaccia2}.
\end{proof}
\begin{proof}[Proof  of Proposition \ref{schiaccia2},  item 2]
We need to show that, for any
$j,k\in\{1,\dots,m\}$,  $\abs{w}= s-1 $
we have the estimate $|X_j^\s X_k^\s f_w^\s|\leq C$,
uniformly in   $x\in K$ and $\s\le \s_0$.
Write
 \[\begin{aligned}
  X_{j }^\s X_{k }^\s f_w^\s &=  X_{j }^\s (X_{k } f_w)^{(\s)} +  X_{j }^\s\big(
  X_{k }^\s f_w^\s - (X_{k } f_w)^{(\s)}
 \big)
= : M+N.   \end{aligned}
\]
Now, letting $\phi_\s(\xi)=\s^{-n}\phi(\xi/\s)$, we have
\[\begin{aligned}
  M(x)&=  (f_{j}^\a)^{(\s)}(x)\p_{x_\a}\int X_k f_w   (x-\s y)\phi(y) dy
  \\&  = -\int    (f_{j}^\a)^\s(x)   X_k f_w (z) \p_{z_\a} (\phi_\s(x-z))dz
  \\&=-  \int
    f_{j }^\a(z)  X_k f_w (z)\p_{z_\a}  (\phi_\s(x-z)) dz
  \\& \quad +
   \int \big\{(f_{j}^\a)^\s(x)-
    f_{j}^\a(z) \big\} X_k f_w (z) \frac{(\p_\a \phi)_\s(x-z)}{\s} dz.
  \end{aligned}
\]
The first line can be estimated integrating by parts  by means of \eqref{chillo}. The
estimate of the second line follows from the Lipschitz continuity of the functions $f_i$.

Next we control $N$.
 \[
  \begin{aligned}
   N(x)&= (f_{j}^\a)^{(\s)}(x) \p_{x_\a}\Big\{  X_k^\s\int f_w(x-\s y)\phi(y)dy -
  \int  (X_k f_w)(x-\s y)\phi(y) dy \Big\}
\\&=(f_{j}^\a)^{(\s)}(x) \p_{x_\a}\Big\{ \int (f_k^\b)^\s(x)\p_\b f_w(z)\phi_\s(x-z)dz
\\&\qquad\qquad \qquad\qquad - \int f_k^\b(z) \p_\b f_w(z)\phi_\s(x-z)dz\Big\}
\\&= (f_{j}^\a)^{(\s)}(x)\int\Big\{ (\p_\a f_k^\b)^\s(x) \p_\b f_w (z)\phi_\s(x-z)
\\&\qquad \qquad \qquad
+ [(f_k^\b)^\s(x) - f_k^\b(z)]\p_\b f_w(z)\frac{1}{\s}(\p_\a \phi)_\s(x-z) \Big\}dz
\end{aligned}
 \]
The estimate is  concluded, because $\p_\b f_w$ is bounded, while
 $|(f_k^\b)^\s(x) - f_k^\b(z)|\le C\sigma$.
 \end{proof}



\end{document}